\documentclass[12pt,oneside]{amsart}
\usepackage{amssymb, amscd, amsmath, amsthm, latexsym, hyperref,amsrefs}
\usepackage{pb-diagram, fancyhdr, graphicx, psfrag, enumerate}
\usepackage[text={6.2in,8.5in},centering,letterpaper,dvips]{geometry}
\usepackage{enumitem}\setlist[enumerate,1]{font=\upshape, itemsep=.5ex}\setlist[itemize,1]{font=\upshape, itemsep=.5ex}
\usepackage{pinlabel}
\usepackage{xcolor}

\def\Z{{\mathbb Z}}
\def\F{{\mathbb F}}
\def\Q{{\mathbb Q}}
\def\R{{\mathbb R}}

\def\co{\colon\thinspace}

\def\calq{\mathcal{Q}}

\def\cals{\mathcal{S}}

\def\calc{\mathcal{C}}
\def\calm{\mathcal{M}}
\def\calg{\mathcal{G}}

\def\calb{\mathcal{B}}

\def\cs{\mathbin{\#}}
\DeclareMathOperator{\cfk}{\rm CFK}

\newcommand{\spinc}{\ifmmode{{\mathfrak s}}\else{${\mathfrak s}$\ }\fi}
\newcommand{\spinct}{\ifmmode{{\mathfrak t}}\else{${\mathfrak t}$\ }\fi}
\newcommand{\spincr}{\ifmmode{{\mathfrak r}}\else{${\mathfrak r}$\ }\fi}

\def\defeq{\mathrel{\!\mathop:}=}

\def\obF{\overline{\F}}
\def\wM{\widetilde{M}}
\def\wW{\widetilde{W}}
\def\wY{\widetilde{Y}}
\def\wK{\widetilde{K}}
\def\wX{\widetilde{X}}
\def\wSig{\widetilde{\Sigma}}

\newtheorem{theorem}{Theorem} [section]
\newtheorem{lemma}[theorem]{Lemma}
\newtheorem{corollary}[theorem]{Corollary}
\newtheorem{proposition}[theorem]{Proposition}

\theoremstyle{definition}

\newtheorem{definition}[theorem]{Definition}
\newtheorem{example}[theorem]{Example}

\begin{document}
\title[Critical points in knot cobordisms]{Critical point counts in knot cobordisms: abelian and metacyclic invariants}

\author{Charles Livingston}
\thanks{This work was supported by a grant from the National Science Foundation, NSF-DMS-1505586.   }
\address{Charles Livingston: Department of Mathematics, Indiana University, Bloomington, IN 47405}\email{livingst@indiana.edu}



\begin{abstract}   For a pair of knots $K_1$ and $K_0$, we consider the set of four-tuples  of integers $(g, c_0,c_1, c_2)$  for which there is a cobordism from $K_1$ to $K_0$ of genus $g$ having  $c_i$ critical points of each index $i$.  We describe   basic properties that such sets must satisfy and then build homological obstructions to membership in the set.  These obstructions are determined by   knot invariants arising from cyclic and metacyclic covering spaces.
\end{abstract}

\maketitle


\section{Introduction}  Given a pair of  knots $K_1$ and $ K_0$ in $S^3$, let $\calg(K_1,K_0)$ denote the set of all four-tuples $(g, c_0,c_1, c_2)$ of nonnegative integers  for which there is a smooth orientable cobordism from $K_1$ to $K_0$ of genus $g$ having  $c_i$ critical points of each index $i$.   Our goal is to identify ways in which classical knot theory   can provide constraints on this set.   The value of $c_1$ is determined by those of $g, c_0$ and $c_2$, so our investigation is reduced studying the sets $\calg_g(K_1, K_0)$ consisting of nonnegative pairs $(c_0, c_2)$ for which there is a genus $g$ cobordism from $K_1$ to $K_0$ having $c_0$ and $c_2$ critical points of index 0 and 2, respectively.

  A number of well-studied problems can be formulated in terms of $\calg(K_1,U)$, where $U$ is the unknot: related topics include the knot four-genus, the slice-ribbon conjecture, problems related to the ribbon-number of ribbon knots, and general  unknotting operations.   The set  $\calg_0(K_1,K_0)$ is related to knot concordances and in particular to the existence and properties of ribbon concordances.  
Papers  that touch  on aspects of these topics include~\cites{MR3307286, MR4024565, Sarkar_2020, MR3825859, MR4186142, MR968881, hom2020ribbon, MR704925, gong2020nonorientable, friedl2021homotopy, MR634459, MR1075165, MR2262340, MR2755489, MR4017212}.  Through the use of   cyclic branched covers, this study  is related to the study of the handlebody structure of cobordisms between three-manifold, as presented, for instance, in~\cite{MR3825859}. 

We have several  goals.  The first is simply to present this perspective on knot cobordism.  Next, we describe how homological invariants of cyclic branched covers of knots provide constraints on the sets $\calg(K_1,K_0)$; this work consists of    extensions of known results concerning ribbon disks and concordances to the setting of cobordisms.  Our use of equivariant homology groups lets us further refine our results.   After this, we consider the use of metacyclic invariants; these arise from    cyclic covers of cyclic branched covers.  Finally, we list some problems that arise from this perspective.

\medskip

\noindent{\bf Summary of results.}  In seeking invariants from  $M_n(K)$,  the $n$--fold branched cover of a knot $K$, or from a $q$--fold cyclic cover of  $M_n(K)$, one faces a series of choices: the values of $n$ and $q$; the coefficients $\F$ for the homology groups; and  the choice of which $q$--fold cover to consider.  There is also a decision as to whether to take into account the module structure of the homology, viewed as an $\F[\Z_n]$--module or $\F[\Z_q]$--module.  As has been done in the past, we will   follow a path that is sufficiently complicated  to illustrate the techniques but is simple enough  to avoid technicalities.   For instance, we will work with knots for which the associated    $\F[\Z_n]$--modules are of a simple form.  

Our main result that is based on cyclic branched covers is the following.

\smallskip

\noindent{\bf Theorem~\ref{thm:homology-bound}.} {\it Suppose that $\Sigma$ is a cobordism from $K_1$ to $K_0$.    Then for all $n$, for all prime powers  $p$ satisfying $p - 1 \equiv 0 \mod n$, and for all $\zeta \in \F_p$ satisfying $\zeta^n =1$, we have 
\[
c_0(\Sigma) \ge \frac{\beta_1^\zeta( M_n(K_1),\F_p)  - \beta_1^\zeta( M_n(K_0), \F_p)}{2} -  g(\Sigma).
\]
}
\smallskip 

\noindent In this statement, $\beta_1^\zeta( M_n(K_1),\F_p)$ is the dimension of the $\zeta$--eigenspace of the $\Z_n$--action on $H_1(M_n(K),\F_p)$, where $\zeta \in \F_p$ is an  $n$--root of unity in the finite field with $p$ elements.  Averaging over the set of $n$--roots of unity yields the following simplier, but often weaker, result. 
\smallskip

\noindent{\bf Corollary~\ref{corollary:bound}.}  Under the conditions of Theorem~\ref{thm:homology-bound},
\[
c_0(\Sigma) \ge \frac{\beta_1( M_n(K_1),\F_p)  - \beta_1( M_n(K_0),\F_p)}{2(n-1)} -  g(\Sigma).
\]

A simple application of Corollary~\ref{corollary:bound} concerns  $3$--stranded pretzel knots:  $P_k =  P( 2k+1, -2k-1, 2k+1)$.  These are ribbon knots.  It follows from   Corollary~\ref{corollary:bound}   that  that if $2i + 1$ and $2j +1$ are distinct primes, then there is a genus $g$ cobordism from $\alpha P_i$ to $\beta P_j$ having with $c_0 \ge 0$ and $c_2 \ge 0$ critical points of index $0$ and $2$, respectively, if and only if $c_0 \ge  \alpha - g$ and $c_2 \ge \beta-g$.  This is proved using 2--fold branched covers.

We will also present an example that  depends on the full strength of Theorem~\ref{thm:homology-bound}, using     higher-fold covers and the eigenspace splitting.  The example is    built from the knot $10_{153}$, which is a ribbon knot with ribbon number 1 (see~\cite{MR1417494}).  We show that there exists a genus   $g$ surface  in $B^4$ bounded by $\alpha 10_{153}$ having $c_0$ and $c_2$ index 0 and index 2 critical points, respectively, if and only if $c_0 \ge \alpha+1 - g$.

Examples in which metacyclic covers yield stronger results will be built from knots $K(k, J)$  illustrated in Figure~\ref{fig:KkJ2}.   In the figure, the right band is tied in the knot $J$.  The left band has $-k$ full twists and the right band had $k+1$ full twists.  If $J$ is a ribbon knot, then this knot is ribbon: the simple closed curve that goes over each band once in opposite directions has framing 0 and has knot type $J$.  This family is of interest because the Seifert form of $K(k,J)$ is  independent of the choice of $J$, and thus no homological invariants arising from cyclic branched covers can be used to distinguish a pair $K(k, J_1)$ and $K(k, J_2)$.  However, the branched cyclic covers $M_n(K(k,J))$   themselves have cyclic covers, and the homology of these iterated covers does depend on $J$.  
In Section~\ref{sec:6_1} we will explore these examples in detail, focusing on the case of $k =1$ and $J$ is a multiple of either $K(1,U) = 6_1$ or $K(2, U)=10_3$. The obstructions we develop are determined from  3--fold cyclic covers of the 2--fold branched cover of  $S^3$, but the proofs of the results require that we consider covers of order $3^b$ for some unknown value of $b$.  This is a reflection of an underlying issue that first appeared in~\cite{MR900252}. 

\begin{figure}[h]
\labellist
 \pinlabel {\text{{$J$}}} at 395 190
  \pinlabel {$x$} at 115 175
    \pinlabel {$y$} at 225 175
\endlabellist
\includegraphics[scale=.37]{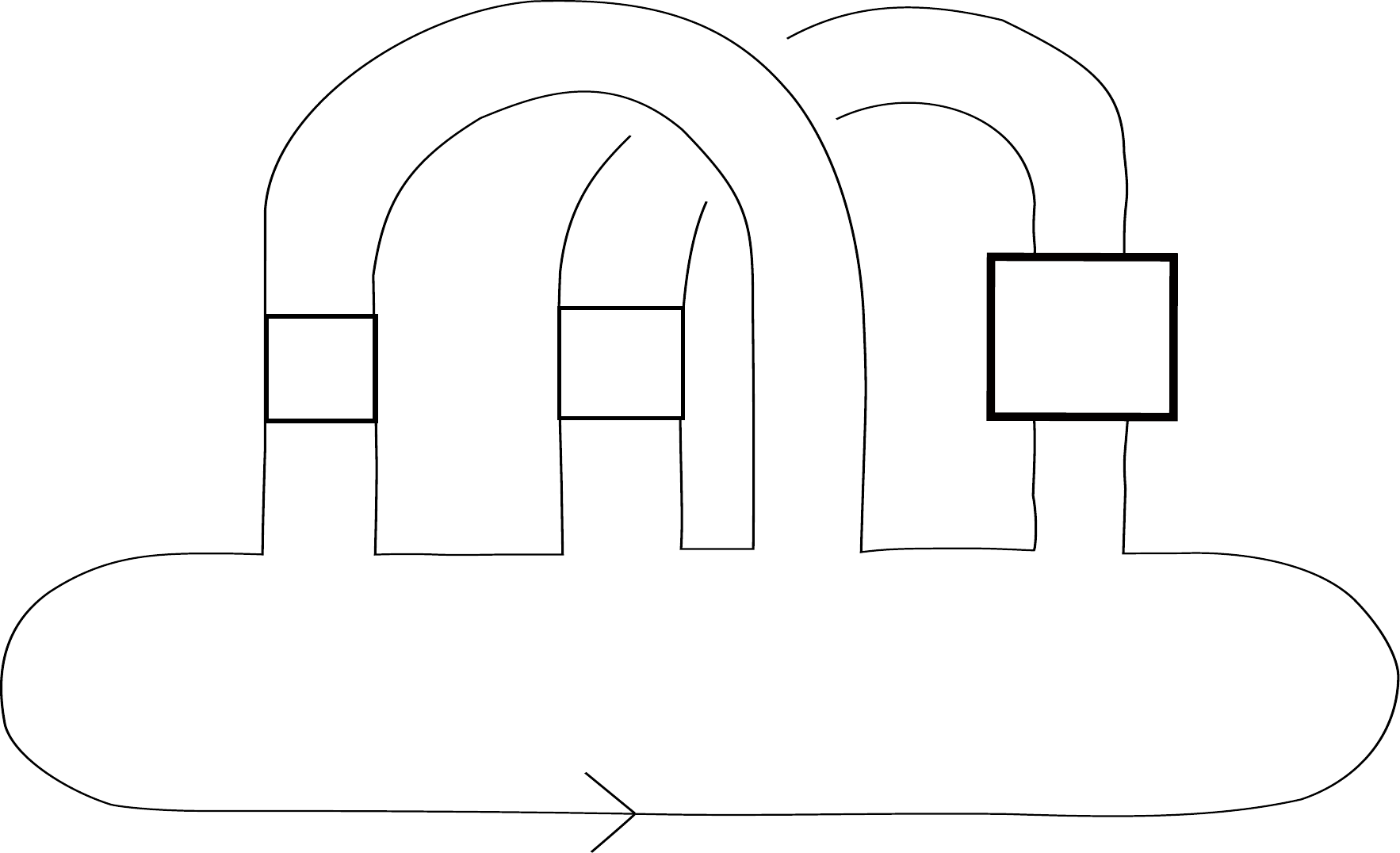} 
\caption{Basic knot $K(k,J)$.  The right band has the knot $J$ tied in it;  $x = -k$ and $y = k+1$ denote full twists.}
\label{fig:KkJ2}
\end{figure}

\smallskip

\noindent{\it Acknowledgments}  Pat Gilmer provided me with many helpful comments that greatly improved the content and exposition of this paper.


\section{The set $\calg_g(K_1, K_0)$.}  

In this section, we present in detail the knot invariants of interest  and describe some of their basic properties.

\subsection{The definition of $\calg_g(K_1, K_0)$}

We view knots  as  smooth oriented diffeomorphism classes of  pairs $(S,K)$ where $S$ is  diffeomorphic to $S^3$ and $K$ is  diffeomorphic to $S^1$.  We will be using  the shorthand  notation  $K \subset S^3$ or simply  $K$  for such a pair;   $-K$ denotes the pair $(-S, -K)$.       A cobordism from a knot $K_1 $ to a knot  $K_0$  consists of a smooth oriented surface $\Sigma \subset S^3 \times [0,1]$ for which $\partial (  S^3 \times [0,1] , \Sigma) = -(S^3,K_0)    \bigsqcup (S^3, K_1)$.   (In particular, $\Sigma \cap (S^3 \times \{1\})= K_1$.) We will assume that $\Sigma$ is {\it connected}.  We will also restrict our attention to {\it Morse cobordisms}, those for which the  projection $\Sigma \to [0,1]$ is a Morse function.

Viewing  $\Sigma$ as a twice punctured surface of genus $g$, we have that  $\beta_1(\Sigma) = 2g+1$; alternatively, $g  = (\beta_1(\Sigma) -1) /2$.  Will write $g(\Sigma)$ for the value of $g$.

We let $c_0(\Sigma), c_1(\Sigma),$ and $c_2(\Sigma)$ denote the number of local minima, saddle points, and local maxima of the projection of $\Sigma$ to $[0,1]$, respectively.    The height function on  $\Sigma$ determines a handlebody structure on $(\Sigma, K_0)$ having $c_0$, $c_1$, and $c_2$ handles of dimensions 0, 1, and 2, respectively.  We will move between the Morse function and the handlebody decomposition without further comment.

If $g(\Sigma) = 0$, then $\Sigma$ is called a {\it concordance}. If $c_2(\Sigma) = 0$, then $\Sigma$ is called a {\it ribbon cobordism}. 

An Euler characteristic argument shows that for a genus $g$ cobordism with $c_0$, $c_1$, and $c_2$ critical points of each index, we have $c_1 = c_2 +c_0 +2g  $.  Thus, to understand the counts of  critical points of possible cobordisms, or equivalently the number of handles in the  corresponding handlebody structure, we do not need to keep track of the value of $c_1$.   (Many past papers focus on $c_1$,  for instance in studying the  ribbon number of ribbon knots,   but notice that if there is a cobordism from $K_1$ to $K_0$ with $c_1$ saddle points, there is also a cobordism from $K_0$ to $K_1$ with $c_1$ saddle points;  we can more readily highlight the  asymmetry of the general problem by using $c_0$ and $c_2$.)

\begin{definition}  For knots   $K_1$ and  $K_0$,   set 
\begin{itemize}

\item  $\calg_g(K_1,K_0) = \{ (  c_0(\Sigma), c_2(\Sigma))\ \big| \ \text{$\Sigma$  is a   cobordism from $K_1$ to $K_0$ with $g(\Sigma) = g$} \} \subset  (\Z_{\ge 0})^2$.


\item  $\calg(K_1,K_0) = \{ ( g,  c_0  , c_1,  c_2 )\ \big| \  (c_0, c_2) \in \calg_g(K_1,K_0) \ \text{and}\  c_1 = c_2 + c_0 +2g \} \subset  (\Z_{\ge 0})^4$.

\end{itemize}
\end{definition}

\subsection{Elementary  properties of   $\calg_g(K_1, K_0)$.}

We begin with the following proposition, which is   no more  than a restatement of the definition of ribbon cobordism.

\begin{proposition} There exists a $c_0 \ge 0$ such that $(c_0,0) \in  \calg_g(K_1, K_0)$ if and only if there exists a genus $g$ ribbon cobordism from $K_1$ to $K_0$.
\end{proposition}

A cobordism can be modified by adding a pair of critical points of indices 0 and 1, or of indices 1 and 2, without altering the genus. Thus we have the next result.

\begin{proposition}\label{prop:shift1} For a pair of knots $K_1$ and $K_0$, if  $(c_0, c_2) \in \calg_g(K_1,K_0) $, then
$(  c_0+i , c_2+j ) \in \calg_g(K_1,K_0) $ for all $i, j \ge 0$.

\end{proposition}

It follows that each $\calg_g(K_1,K_0)$ is a finite  union of quadrants, $\bigcup_\alpha \calq(a_\alpha, b_\alpha)$, where  
\[ \calq(a ,b  ) \defeq \{ (i,j) \ \big| \  i \ge a   \text{\ and \ }  j \ge b  \ \}.\] 
Figure~\ref{fig:graphs1}  illustrates the union of quadrants $\calq(2, 3) \bigcup \calq(5,1)$.\begin{figure}[h]
\labellist
\endlabellist
\includegraphics[scale=.16]{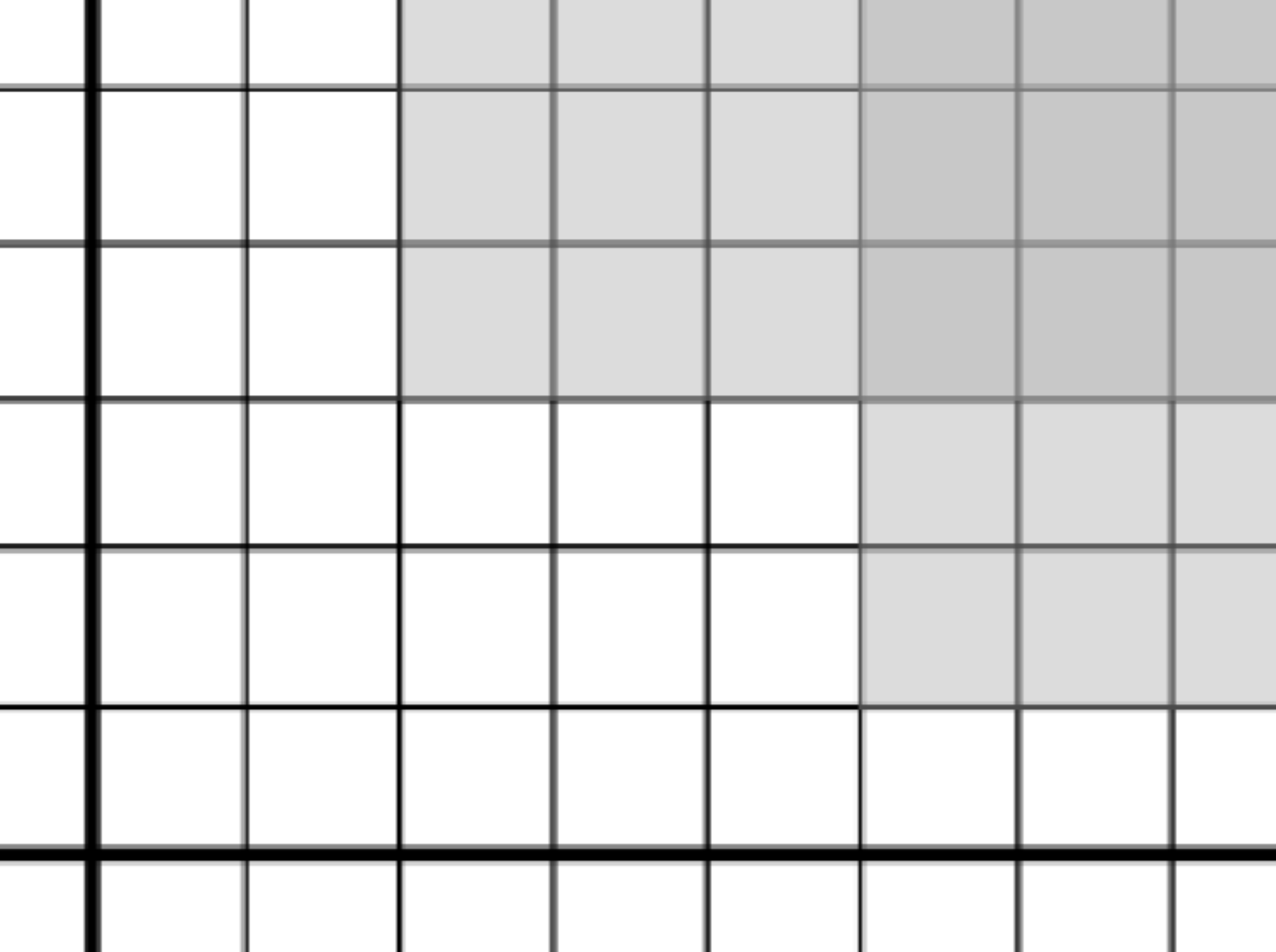} 
\caption{Graph of $\calq(2,3) \bigcup \calq(5,1)$.}
\label{fig:graphs1}
\end{figure}

If for some pair of knots $K_1$ and $K_0$ and $g\ge0$, the graphic in Figure~\ref{fig:graphs1} represents $\calg_g(K_1,K_0)$, then the fact that there are no point on either axis implies that there does not exist a  genus $g$ ribbon cobordism from $K_1$ to $K_0$ or from $K_0$ to $K_1$.

Next, we observe the most basic ways in which points in $\calg_g(K_1, K_0)$ determine points in  $\calg_{g+1}(K_1, K_0)$

\begin{proposition}\label{prop:shift2} For a pair of knots $K_1$ and $K_0$, suppose that $(c_0, c_2) \in \calg_g(K_1,K_0) $.   
\begin{enumerate}

\item  If $c_0 >0$, then $ ( c_0- 1 , c_2  ) \in \calg_{g+1}(K_1,K_0) $.

\item  If $c_2 >0$, then $ ( c_0  , c_2  - 1) \in \calg_{g+1}(K_1,K_0) $.

\end{enumerate}

\end{proposition}

\begin{proof}  In terms of cross-sections of the cobordism, an index 0 critical point at height $t$ corresponds to the addition of an unknotted, unlinked component to the cross-section of $\Sigma$ as the height increases past $t$. The same addition can be realized by performing a trivial band move to the cross-section at height just below $t$.  This corresponds to adding a   critical point of index 1  in exchange for eliminating the  index 0 critical point.  It increases the genus by 1.  A similar construction eliminates index 2 critical points. 
\end{proof}

\begin{example}  Figure~\ref{fig:wedge2} illustrates how  a  point in $\calg_g$ generates points in $\calg_{g+1}$. In this example, the point $(4,2) \in \calg_0$.  Using Propostion~\ref{prop:shift2} we see that $\{ (3,2), (4,1)\} \subset   \calg_1$.      This in turn implies that  $\{ (2,2), (3,1), (4,0)\} \subset   \calg_2$.  It next follows that $\{ (1,2), (2,1), (3,0)\} \subset   \calg_3$.  As a consequence, we have  $\{ (0,2), (1,1), (2,0)\} \subset   \calg_4$ and then that  $\{ (0,1), (1,0)\} \subset   \calg_5$.   Finally, $(0,0) \in \calg_g$ for all $g \ge 6$.  

In this example, if the first figure represents $\calg_0 $ for some pair of knots, we are not asserting the remaining diagrams illustrate the $\calg_g$, but only that they represent subsets of the $\calg_g$.   Example~\ref{ex:genus-grow} in Section~\ref{sec:homconstrain} we will show that $\calg(K_1,K_0)$ can be strictly larger than the set guaranteed by Proposition~\ref{prop:shift2}.

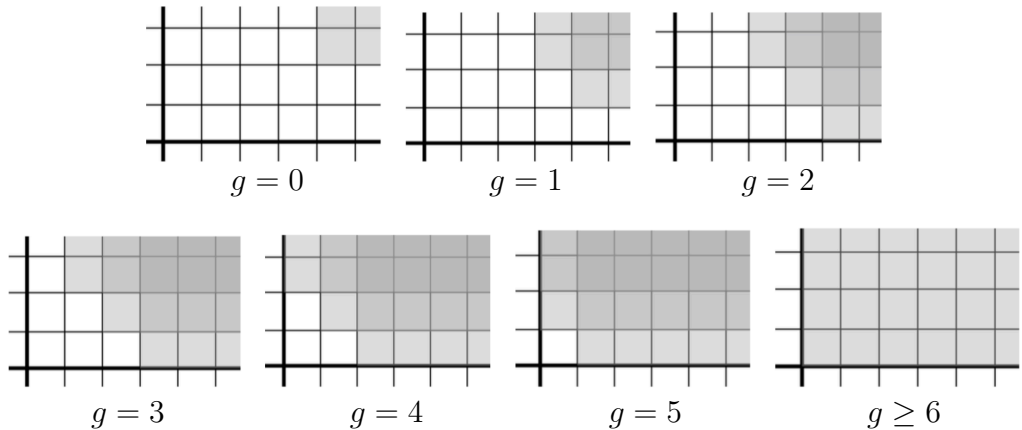
\begin{figure}[h]
\labellist
 \pinlabel {$g=0$} at 390 350  
 \pinlabel {$g=1$} at 780 350
   \pinlabel {$g=2$} at 1160 350
    \pinlabel {$g=3$} at 180 000
     \pinlabel {$g=4$} at 570 000
      \pinlabel {$g=5$} at 960 000
       \pinlabel {$g\ge 6$} at 1350 000
\endlabellist
\includegraphics[scale=.25]{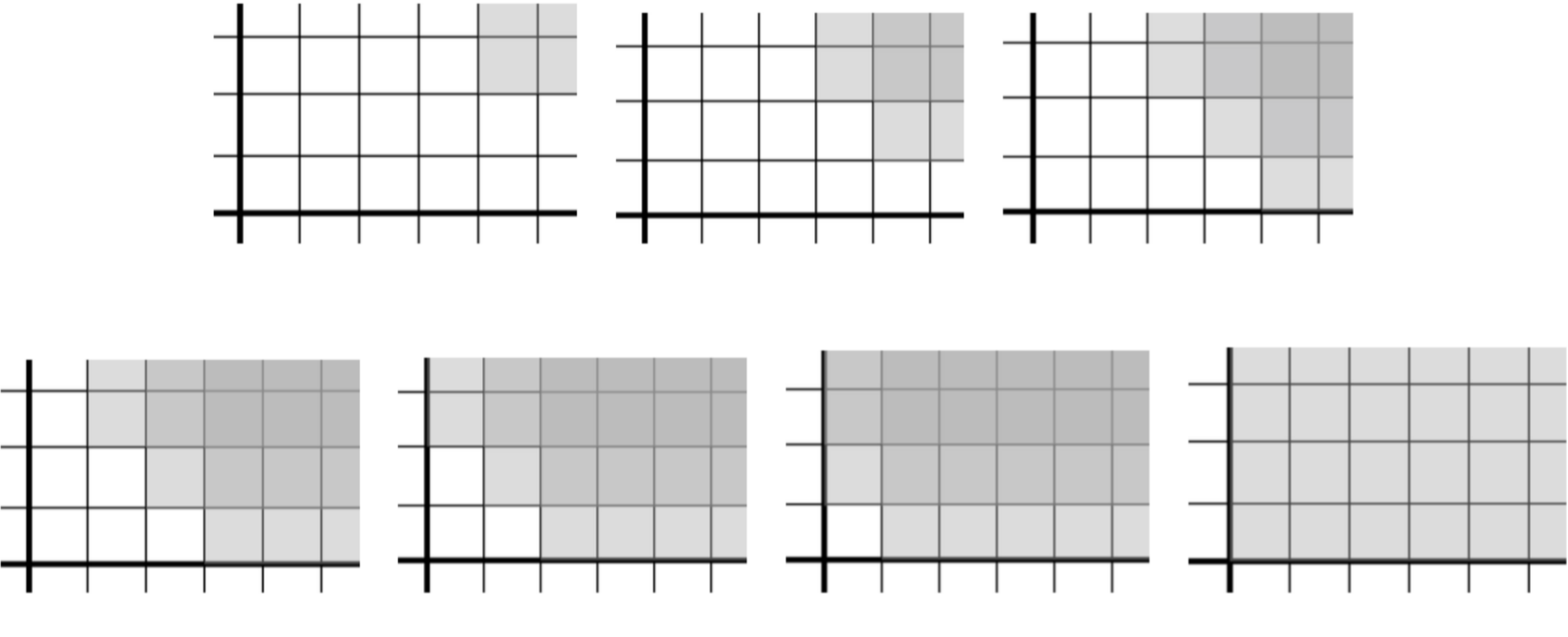} 
\caption{Possible sets $\calg_g(K_1,K_2)$.}
\label{fig:wedge2}
\end{figure}

\end{example}

\subsection{The  set of $\calg(K_1,K_2)$ and the associated sequence.}
It is apparent that each $\calg_g$ is determined by a unique finite set of points and that for large $g$, $\calg_g$ consists of the entire quadrant.  This is summarized in the following theorem. 

\begin{theorem}  Each  set  $ \calg(K_1, K_0)$ is   determined by a finite  sequence 
\[ \cals(K_1, K_0) = \big(  (g_1, a_1, b_1),  (g_2, a_2, b_2), (g_3, a_3, b_3), (g_4, a_4, b_4), \ldots,  (g_k, 0, 0)  \big)
\]
of elements in $(\Z_{\ge 0})^3$ which is lexicographically  ordered.    There is a unique minimal length such sequence.
\end{theorem}

As an example, some of the terms of the lexigraphically ordered sequence corresponding to the regions in Figure~\ref{fig:wedge2} are 
\[
\big(  (0,4,2),(1,3,2),(1,4,1),(2,2,2),(2,3,1),(2,4,0),   \ldots, (5,0,1), (5,1,0), (6,0,0)   \big).
\]
A general problem that seems to be  beyond   currently available techniques is to determine if there are any constraints on the sequences  that can arise from a pair of knots  other than those that are a consequence of Propositions~\ref{prop:shift1} and~\ref{prop:shift2}.     For instance, the ribbon conjecture  can be stated as the following:  if $(0,c_0,c_2) \in \cals(K, U)$ for some $c_0$ and $c_2$, then   $(0, c_0', 0) \in  \cals(K, U)$ for some $c_0'$.  The generalized ribbon conjecture states that if $ (g,c_0,c_1) \in \cals(K, U)$ for some $g, c_0$ and $c_2$, then $ (g,c_0',0) \in \cals(K, U)$ for some $c_0'$.  See Section~\ref{sec:problems} for a further discussion.  

\subsection{The case of $K_0 $ is unknotted.}
Understanding $\calg_g(K, U)$ is   equivalent to analyzing surfaces bounded by $K$ in $B^4$.   
Given a knot $K \subset S^3$, we let $\Sigma \subset B^4$ with $\partial \Sigma = K$.  We will assume the radial function is Morse on $\Sigma$; hence, we can define the count of critical points as before.

\begin{definition}  For knot a knot $K$, set 
\[
\calb_g(K ) = \{ (  c_0(\Sigma),  c_2(\Sigma))\ \big| \   \Sigma \subset B^4, \   \partial \Sigma = K,      \text{\  and\ } g(\Sigma) = g\}.
\]
\end{definition}
The following is clear. 
\begin{proposition} For any knot $K$,  $(c_0, c_2) \in \calb_g(K) $ if and only if $(c_0 -1, c_2) \in \calc_g(K, U)$.
\end{proposition}

The sets $\calg_g(K_1, K_0)$ and $\calb_g(K_1 \cs -K_0)$ are   related, but note that in considering $\calg_g( K_1 \cs -K_0)$ we have lost the asymmetry of the general problem.   Let $\text{b}(K)$  denote the minimum number of index 0 critical points in a ribbon disk for $K\cs -K$.  This invariant is related to classical three-dimensional knot invariants.  For instance, let $\text{br}(K)$ denote the bridge index of $K$.  A ribbon disk for $K \cs -  K$ with $c_0  = \text{br}(K)$ and $c_1  = \text{br}(K) - 1$ is easily constructed;  thus  $\text{b}(K) \le  \text{br}(K)$.  Results concerning the interplay between these invariants appears in~\cite[Section 1] {MR4186142}. See also Problem~\ref{prob1} in Section~\ref{sec:problems}.

Given a cobordism from $K_1$ to $K_0$, we can start with the ribbon surface for $K_0 \cs -K_0$ to build a slicing surface for $K_1 \cs -K_0$: use the cobordism to change  $K_1 \cs -K_0 $ into $ K_0 \cs -K_0$ and then attach  a   slice disk.  This leads to the next result.

\begin{theorem}  If $(  c_0,   c_2) \in \calg_g(K_1, K_0)$, then $(  c_0 +\text{\upshape{b}}(K_0) ,   c_2) \in \calb_g(K_1\cs -K_0)
$
\end{theorem}

In the reverse direction, given a surface bounded by $K_1 \cs -K_0$, we can build a cobordism from $K_1 $ to $K_0$: build a cobordism from $K_1$ to $K_1 \cs -K_0 \cs K_0$ and then cap it  off with the surface bounded by $K_1 \cs -K_0$.  This yields the following.

\begin{theorem}  If $(  c_0, c_2) \in \calb_g(K_1\cs -  K_0)$, then $\left(  c_0  -1      , c_2 +\text{\upshape{b}}(K_0)\right) \in \calg_g(K_1,  K_0)$.
\end{theorem}


\section{\bf Covering spaces  and equivariant knot theory}\label{sec:equivariant}

In this section, we set up the notation for covering spaces and the general theory of the associated equivariant homology theory.  We then consider a technical issue that arises from the following situation.  A homomorphism  $\rho \co \pi_1(X) \to \Z_m$ determines a homomorphism   $\rho \co \pi_1(X) \to \Z_{km}$ for any $k$ via inclusion; we will need to understand relationships between the equivariant homology groups of the associated $m$--fold and $km$--fold cyclic covers. 

\subsection{Cyclic covers of knots}  

Let $K\subset S^3$ be a knot and let $\Sigma \subset S^3 \times [0,1]$ be a cobordism between knots.

\begin{definition}$\ $
\begin{itemize}
\item $M_n(K)$ will denote  $n$--fold cyclic  cover of $S^3$ branched over  $K$.
\item   $\wK$ denotes the preimage of $K$ in $M_n(K)$.  
\item  $W_n(\Sigma)$  and  $\wSig$ denote the $n$--fold cyclic  cover of  $S^3 \times [0,1]$ branched over $\Sigma$ and the preimage of $\Sigma$. 
\item  $M_\infty(K)$ and $W_\infty(\Sigma)$ will  denote the infinite cyclic covers of $S^3 \setminus K$ and   $\big(S^3 \times [0,1]\big) \setminus \Sigma$.
\end{itemize}
\end{definition}

\subsection{Covering space theory.}   For any group $\Gamma$, let $K_\Gamma$  denote an Eilenberg-MacLane space for $\Gamma$ and let $E_\Gamma$ denote its universal cover.  All spaces $X$ considered here will be  connected manifolds and covering spaces will be  abelian, so we need not discuss details about the underlying point set topology and basepoint issues.  

If $X$ is connected and $\rho \co \pi_1(X)  \to \Gamma$ is a homomorphism, then it induces a map $X \to K_\Gamma$.  The pullback of $E_\Gamma \to K_\Gamma$ to $X$ is a covering space $\widetilde{X}_\rho$.  Points in the preimage of a basepoint in $\widetilde{X}_\rho$ correspond to elements of $\Gamma$ and components of  $\widetilde{X}_\rho$ corresponds to cosets of $\rho(\pi_1(X))\subset \Gamma$.

\subsection{Equivariant homology and Betti numbers}

For any CW--complex  $X$, suppose that $T$ is an order $m$ homeomorphism that preserves the CW--structure.  Let $\F$ be a separable field, for instance $\Q$, $\R$, or a finite field $\F_p$.  Let $\F$ have   algebraic closure $\obF$.  The homology group $H_i( X, \obF)$ splits into $m$ eigenspaces, $H_i^\zeta(X, \obF)$ for the $m$ distinct $m$--roots of unity $\zeta \in \obF$.  These eigenspaces are isomorphic to the homology groups associated to an  eigenspace splitting of the CW--chain complex.   We define $\beta_i^\zeta(X, T,     \F)$ to be the dimension of $H_i^\zeta(X, \obF)$.  If $\zeta_1$ and $\zeta_2$ are Galois conjugate, then  $\beta_i^{\zeta_1}(X, T,     \F) =  \beta_i^{\zeta_2}(X, T,     \F)$.

If  $\rho \co \pi_1(Y) \to \Z_m$ is a homomorphism,  then there is an induced $m$--fold covering space $\widetilde{Y}_\rho$ with canonical deck transformation $T_\rho$.  We will sometimes  highlight the role of $\rho$ in our notation by  writing 
$\beta_i^\zeta(\wY_\rho, \rho,  \F)$ for
$\beta_i^\zeta(\widetilde{Y}_\rho,T_\rho , \F)$.  In our applications, the 
space $Y$ will be either $M_n(K)$   or $W_n(\Sigma)$.  The corresponding covers $\widetilde{Y}_\rho$ are called {\it metacyclic} branched  covers of $K$. In the case that $n=2$, the cover $\wM_2(K)$ is what is called a regular $2m$--fold {\it  dihedral cover} of $S^3$ branched over  $K$.

\subsection{Relations between equivariant Betti numbers}
For a connected manifold $X$, suppose that $\rho\co \pi_1(X) \to \Z_m$ is a homomorphism.  Let $\rho' \co \pi_1(X) \to \Z_{km}$ be induced by the inclusion $\Z_m \subset \Z_{km}$.

\begin{theorem}\label{thm:disjoint} With the conditions given above, the induced $km$--fold cover of $X$  is the disjoint union of $k$ copies of the $m$--fold cover of $X$:
\[ \widetilde{X}_{\rho'} \cong  \widetilde{X}_\rho \sqcup  \widetilde{X}_\rho \sqcup \cdots  \sqcup \widetilde{X}_\rho.
\]
The order $km$ deck transformation shifts each summand to the next.  The last summand is mapped to the first via the order $m$ deck transformation of $\widetilde{X}_\rho$.
\end{theorem}

\begin{proof}  This result follows from standard covering space theory.  \end{proof}

\begin{theorem} \label{thm:equivcount} Suppose that $\rho \co \pi_1(X) \to \Z_m$ is a homomorphism  and $\rho' \co \pi_1(X)\to \Z_{km}$ is the composition of $\rho $ with the inclusion $\Z_m \subset \Z_{km}$.   Let $T_\rho$ be the order $m$ deck transformation of $\widetilde{X}_{\rho}$ and let $T_{\rho'}$ be the order $km$ deck transformation of  $\widetilde{X}_{\rho'}$. Then the $k$ power of $T_{\rho'}$ is a transformation of order $m$ and 
\[\beta_i^\zeta (\widetilde{X}_{\rho'}, T_{\rho'}^{k}, \F) = k \beta_i^\zeta(\widetilde{X}_\rho, T_\rho, \F).
\]

\end{theorem}

\begin{proof}  The action of the  $k$ power of $T_{\rho'}$ leaves invariant each factor $\widetilde{X}_\rho$  in    the decomposition   given by Theorem~\ref{thm:disjoint},  $\widetilde{X}_\rho \sqcup  \widetilde{X}_\rho \sqcup \cdots  \sqcup \widetilde{X}_\rho$.  It restricts to each factor to be the deck transformation of $  \widetilde{X}_\rho$.  
\end{proof}

\subsection{The equivariant CW--chain complex}\label{sec:eqvcw}  Suppose that $X$ has the structure of a CW--complex.  That structure lifts to give a   compatible CW--structure on each covering space; there is also a lifted CW--structure on branched covers, assuming that the branch set in $X$ is a subcomplex.  In particular, there are CW--chain complexes $C_*(\widetilde{X} , \F)$ for any covering space  or branched covering spaces we consider.

We will also be working with pairs  $(X,M)$ that have a relative CW--structure, and these structures also lift to covering spaces.   

If $\rho \co \pi_1(X) \to \Z_m$ is a homomorphism and $\wX_\rho$ is the induced cover with deck transformation $T_\rho$, then for each $m$--root of unity $\zeta \in \overline{\F}$ there is a subcomplex $C_*^\zeta( \wX_\rho, T_\rho, \overline{\F})$.  The homology of this complex is the equivalent homology discussed earlier.

\begin{theorem} \label{thm:cocount} Suppose that $\F$ contains a primitive $m$--root of unity $\zeta$.  Let $X$ be a space or a pair of spaces.  Suppose that $\rho\co \pi_1(X) \to \Z_m$ induces a cover $\wX_\rho$ with deck transformation $T_\rho$.  Let  $\rho' \co   \pi_1(X) \to \Z_{km}$ be the composition of $\rho$ and the inclusion map $\Z_m \subset \Z_{kn}$, and let $\wX_{\rho'}$ be the associated cover with deck transformation $T_{\rho'}$.  
\begin{itemize}
\item  For all $i$,   $ \dim C_i^\zeta (\wX_\rho, T_\rho, \F) = \dim   C_i(X)    $.

\item For all $i$,   $  \dim C_i^\zeta (\wX_{\rho'}, (T_{\rho'})^k, \F) = k \dim  C_i(X)$.

\end{itemize}

\end{theorem}
\begin{proof} Let $t$ generate $\Z_m$.  As a $\F[\Z_m]$--module,  $C_i(\wX_\rho, \F)$ splits as a direct sum of modules isomorphic to  $\F[Z_m]$.  There is one summand for each $i$--cell of $X$.    We then have the decomposition $\F[Z_n] \cong \oplus_{i=0}^{n-1} \F[Z_m]/\left< t- \zeta^i\right>$.  The summand  $\F[Z_m]/\left< t- \zeta^i\right>$ is a $\zeta^i$--eigenspace of the action.  Thus, each $i$--cell of $X$ provides an  eigenvector in $C_i^\zeta(\wX_\rho, T_\rho, \F)$.  The second  statement then follows from Theorem~\ref{thm:equivcount}.
\end{proof}

\subsection{Pairs of spaces}  Let $(X, Y)$ be a CW--pair and let $\rho \co H_1(X) \to \Z_m$.   Then there is an associated covering space pair $\widetilde{(X,Y)}$ and we can consider   the equivariant relative homology groups of this cover.  All the statements in the Section~\ref{sec:eqvcw} above carry over to this relative setting.  

\subsection{Computing the equivariant homology for spaces associated to knots}
For any given knot, the computation of $\beta_i^\zeta(M_n(K),      \F_p)$ is fairly straightforward, using little more that what is covered in, say, Rolfsen's text~\cite{MR0515288}.  The computation of the metacyclic invariants can be   technically challenging; in particular, they are not determined by a Seifert matrix.   For this reason, we will restrict our examples to those for which for which the computation is quickly  accessible.    


\section{Handlebody structure}\label{sec:handles}

\begin{theorem} \label{thm:handles2} The pair $(W_n(\Sigma)\setminus \wSig, M_n(K_0)\setminus \wK)$ has a relative handlebody decomposition with:
\begin{itemize}
\item  $nc_0(\Sigma)$  $1$--handles.
\item  $nc_1(\Sigma) $ $2$--handles.
\item  $nc_2(\Sigma)$   $3$--handles.
\end{itemize}
\end{theorem}

\begin{proof} See, for instance, \cite[Proposition 6.2.1]{MR1707327} for a description of the handlebody structure on $(S^3 \times [0,1])\setminus \Sigma$.  That structure lifts to the covering space.
\end{proof}

\begin{theorem} \label{thm:handles} The pair $( W_n(\Sigma),  M_n(K_0))$ has a relative handlebody decomposition with:
\begin{itemize}
\item  $nc_0(\Sigma)$     $1$--handles.
\item  $nc_1(\Sigma) $   $2$--handles.
\item  $nc_2(\Sigma) + 2g(\Sigma)  $     $3$--handles.
\end{itemize}
\end{theorem}

\begin{proof}  We have that $(W_n(\Sigma), M_n(K_0))$ is built from  $(W_n(\Sigma)\setminus \wSig, M_n(K_0)\setminus \wK_0)$ via handle additions.  For each $i$--handle in $\Sigma$ there is an $(i+2)$--handle added.    

The surface  $\Sigma$ can be built with one $0$--cell and $\beta_1(\Sigma)$ $1$--cells.  The 0--cell and the first $1$--cell comprise $K_0$.  Hence, in building  $(W_n(\Sigma), M_n(K_0))$ the   added $2$--handle and the first $3$--handle complete the construction of a product  neighborhood of  $ M_n(K_0 )$.  There remain $(\beta_1(\Sigma) -1)$ $3$--handles to add.  Finally, $\beta_1(\Sigma) -1 = 2g$.
\end{proof}

\section{Homological   constraints arising from cyclic branched covers}\label{sec:homconstrain}

\subsection{Homological constraints.} 

In this section, we will denote the order $n$ deck transformation of $M_n(K)$ by $T$.  That is, no confusion should result by using the symbol $T$ without notating its dependence on $K$ and $n$.  We will work with finite  fields of prime order, $\F_p$, that contain   primitive $n$--roots of unity; that is, $p - 1 \equiv 0 \mod n$.   Unless specified, we  will not assume that a given $n$--root of unity $\zeta$ is primitive.   The main result of this section is the following theorem.

\begin{theorem} \label{thm:homology-bound} Suppose that $\Sigma$ is a cobordism from $K_1$ to $K_0$.    Then for all $n$, for all prime powers  $p$ satisfying $p - 1 \equiv 0\ \mod n$, and for all $\zeta   \in \F_p$ satisfying $\zeta^n =1$, we have 
\[
c_0(\Sigma) \ge \frac{\beta_1^\zeta( M_n(K_1),T, \F_p)  - \beta_1^\zeta( M_n(K_0),T, \F_p)}{2} -  g(\Sigma).
\]

\end{theorem}

Before proving this, we isolate the case $\zeta =1$ in a lemma  and then prove another   lemma that will simplify our exposition.

\begin{lemma} \label{lemma:1eigen} Let  $K$ be a knot and let   $\{n,p\}$ be a relatively prime pair.  Then the  1--eigenspace of the deck transformation acting on $H_1(M_n(K), \F_p)$ is trivial.
\end{lemma}
\begin{proof} This is a fairly standard result, the proof of which we   outline.  Let $ f\co M_n(K) \to S^3$ be the branched cover.  Given any cell $d$ in a
 compatible CW--structure on $S^3$, we can choose a lift $\widetilde{d}$ in $ M_n(K)$ and define the transfer to be $\tau (d) = \sum_{i=0}^{n-1} T^i (\widetilde{d})$.  The choice of coefficients ensures that $\tau$ induces an isomorphism to the $1$--eigenspace,  $H_i(S^3, \F_p) \to H_i^{1}(M_n(K), \F_p)$.  The target is thus trivial for $i=1$.
\end{proof}

\begin{lemma} \label{lemma:inequalities1} Let $(W,M)$ be a CW--pair supporting an action $T$ of $\Z_n$.  Suppose that $\F$ is a field containing  an element $\zeta  \ne 1$ for which $\zeta^n =1$.  Finally, assume that $T$ preserves the components of $M$; that is, that $T_*$ acts trivially on $H_0(M, \F)$.  Then  with $\F$--coefficients, 
\[     \beta_1(W) \le \beta^\zeta_1(M) + \dim( C_1^\zeta(W, M))
\]
and
\[   \beta_1(W) \ge  \beta_1^\zeta(M) + \dim( C_1^\zeta(W, M))   -  \dim( C^\zeta_2(W, M)).\]
\end{lemma}

\begin{proof} Removing  cells of dimension $3$ or higher does not affect any of the terms    in the statement, so we can assume that $W$ is a $2$--complex.  In the proof, to simplify the presentation  we suppress the ``$\F$'' in notation for chain complexes, homology groups, and Betti numbers.

The group $H_0^\zeta(M) = 0$.  Thus, from the long exact sequence, we have
\[ 
H_2^\zeta(W, M) \to H_1^\zeta(M) \to  H_1^\zeta(W) \to H^\zeta_1(W, M) \to 0.
\]
From this it follows that 
\begin{equation}\label{eqn:inequality}
\beta_1^\zeta(W) = \beta_1^\zeta(W, M)   + \beta^\zeta_1(M) - \dim \big(\text{Image}(H_2^\zeta(W, M) \to H_1^\zeta(M)) \big).
\end{equation}

Since  $\beta_1^\zeta(W, M)  \le \dim( C_1^\zeta(W, M))$, the first inequality in the statement of the lemma is immediate.

We have  $\dim \big(\text{Image}(H_2^\zeta(W, M) \to H_1^\zeta(M)) \big) \le \beta_2^\zeta(W,M)$; substituting into Equation~\ref{eqn:inequality} yields
\[
\beta_1^\zeta(W) \ge \beta_1^\zeta(M) + \beta_1^\zeta(W,M) - \beta_2^\zeta(W,M).
\] 
We have that $c_0^\zeta(W,M) = 0$, so a standard Euler characteristic argument implies that $ \beta_1^\zeta(W,M) - \beta_2^\zeta(W,M) =  \dim\big(C_1^\zeta(W,M)\big) - \dim\big(C_2^\zeta(W,M)\big)$.
  Hence,
 \[\beta_1^\zeta(W) \ge \beta_1^\zeta(M) + \dim\big(C_1^\zeta(W,M)\big) - \dim\big(C_2^\zeta(W,M)\big),\] as desired.
\end{proof}

\begin{proof}[Proof of Theorem~\ref{thm:homology-bound}]  

To simplify notation, we let $W = W_n(\Sigma) \setminus \wSig$,  $\partial_0 W =  M_n(K_0)\setminus \wK_0$ and  $\partial_1 W =  M_n(K_1)\setminus \wK_1$. 

 The  $1$--handles  and   $2$--handles in the relative handlebody structure on $(W,\partial_0 W)$  are each freely permuted by the action of the generating deck transformation $T$.  That is,   for $i = 1$ and $i=2$  we have that  the  CW--chain complex $C_i(W,\partial_0 W, \F_p)$ splits as a $\F_p[\Z_n]$--module   into $c_i$ copies of $\F_p[T]/\left< 1 - T^n\right>$.  Each of these  splits into $n$ eigenspaces; letting $\xi$ be a primitive $n$--root of unity, 
\[
\F_p[T]/\left< 1 - T^n \right> \cong  \oplus_{i=0}^{n-1} \F_p[T] / \left< \xi^i - T\right>.
\]
We have that  $\zeta =  \xi^i$ for some $i$, so the $\zeta$--eigenspace of the relative CW--chain complex of $( W, \partial_0 W )$ has  $c_0$ generators in dimension 1 and $c_1$ generators in dimension 2.  That is, $\dim (C_i^\zeta(W, \partial_0 W, \F_p)) = c_i$.   The first inequality of Lemma~\ref{lemma:inequalities1} gives
\[  \beta_1^\zeta(  W, T,  \F_p)\le \beta_1(\partial_0 W,T,  \F_p ) +  c_0(\Sigma).
\]

We have a similar construction of  $W$ starting with $\partial_1 W = M_n(K_1) \setminus \wK_1$.  In this case, we have $\dim (C_1^\zeta(W, \partial_1 W) )= c_2$ and $\dim (C_2^\zeta(W, \partial_1 W) )= c_1$.  Using the second inequality in Lemma~\ref{lemma:inequalities1}, \[  \beta_1^\zeta (W, T, \F_p) \ge \beta_1^\zeta(\partial_1 W,T, \F_p ) +  c_2(\Sigma) -  c_1(\Sigma).
\]
Combining these, we see that
\[ \beta_1^\zeta (\partial_0 W, \F_p)  +  c_0(\Sigma) \ge   \beta_1^\zeta(\partial_1 W, \F) +  c_2(\Sigma) -   c_1(\Sigma) .
\]
Recall  that $c_1(\Sigma) = c_0(\Sigma)+ c_2(\Sigma)  +2g(\Sigma)$.  The previous inequality can be rewritten as
\[  \beta_1^\zeta(\partial_0 W ,T, \F_p) +  c_0(\Sigma) \ge    \beta_1(\partial_1 W,T, \F_p)  +  c_2(\Sigma) -   (   c_0(\Sigma)+ c_2(\Sigma)  +2g(\Sigma)).
\] 
This inequality simplifies to give 
\[
c_0(\Sigma) \ge \frac{\beta_1^\zeta(\partial_1 W,T,  \F_p)  - \beta_1^\zeta(\partial_0 W,T, \F_p)}{2} -  g(\Sigma).
\]
The proof is finished by noting that completing the covers to form the branched cyclic covers adds generators to the CW--complex that are all in the $1$--eigenspace and thus do not change the computation. 
\end{proof}

Early work~\cite{MR704925} studying ribbon     knots provided homological constraints on the structure of the number  the minimum number of index 1 critical points in a ribbon disk based on the homology of the 2--fold  branched covers.    The next theorem is a fairly simple generalization of such results.  Notice that we do not restrict to the ribbon situation, 2--fold covers,  or the case of $g=0$. 

\begin{corollary}\label{corollary:bound}  Under the conditions of Theorem~\ref{thm:homology-bound},
\[
c_0(\Sigma) \ge \frac{\beta_1( M_n(K_1), \F_p)  - \beta_1( M_n(K_0),\F_p)}{2(n-1)} -  g(\Sigma).
\]

\end{corollary}

\begin{proof} The proof consists of summing over the $n-1$  eigenspaces.
\end{proof}
\medskip

\begin{example} 
Let $P_k$ denote the pretzel knot $P( 2k+1, -2k-1, 2k+1)$. These are ribbon knots having Seifert form 
\[
\begin{pmatrix}  
0 & k \\
k+1  & 0 \\
\end{pmatrix}.
\]
Each bounds a ribbon disk with one saddle point and two minimum. We have $H_1(M_2(P_k)) \cong  \Z_{ 2k+1 } \oplus   \Z_{ 2k+1 }$.  

We want to consider the sets $\calg_g(nP_1, mP_2)$ and for convenience  assume that $n \ge m$.   This example presents the case of  $g=0$ and the next considers $g>0$.

Our first observation is that $aJ_k$ bounds a ribbon disk with $a$ saddle points and $a+1$ minima.   From this it is easily seen that there is a concordance $\Sigma$ from $nP_1$ to $mP_2$  with $c_0(\Sigma) = n$, $c_1(\Sigma) = n +m $, and $c_2(\Sigma) = m$.   That is,  $(n,m) \in  \calg_0(nP_3, mP_5)$.

Using $\Z_3$--coefficients in Theorem~\ref{thm:homology-bound}, we see that 
\[ c_0(\Sigma) \ge \frac{2n - 0}{2} -0 = n.
\]
Similar, working with $\Z_5$--coefficients we have $c_2(\Sigma) \ge m$.  Thus, $\calg_0(nP_3, mP_5)$ is precisely the quadrant with vertex $(n,m)$, that is $\calq(n,m)$.

\end{example}

\begin{example}\label{ex:genus-grow}  If $m >0$ and $n>0$,  then  by Proposition~\ref{prop:shift2} we have $\calq(n-1,m) \cup \calq(n, m-1) \subset \calg_1(nP_1, mP_2)$.  Here we show that this is a proper containment, that in fact,  $\calg_1(nP_1, mP_2) = \calq(n-1, m-1)$.  

The construction of a cobordism is simple.  In the initial concordance that we built, the local maxima were at levels below the local  minima.  Because of this, the concordance can be modified by replacing disk neighborhoods of a  maximum point and a minimum point by an annulus near an increasing path from the maximum to the minimum.  The effect is to decrease both $c_0$ and $c_2$ by 1  in exchange for increasing the genus by 1; that is $\calg_1(nP_3, mP_5) \subset \calq(n-1, m-1)$.     Theorem~\ref{thm:homology-bound} immediately  implies that this inclusion must be an equality.

The process can be repeated to prove that for   $g \le m$ we have $\calg_g(nP_1, mP_2) = \calq(n-g, m-g)$.

Finally, Proposition~\ref{prop:shift2} implies that for $m\le g \le n$ we have $\calg_g(nP_1, mP_2) = \calq(n-g, 0)$. For $g\ge n$ we have  $\calg_g(nP_1, mPJ_2) = \calq(0, 0)$.

Figure~\ref{fig:graphic2} illustrates the sets $\calg_g(4P_1, 2P_2)$.

\begin{figure}[h]
 \labellist
 \pinlabel {$g=0$} at 180 00
 \pinlabel {$g=1$} at 500 00
\pinlabel {$g=2$} at 900  00
\pinlabel {$g=3$} at 1250 00
\pinlabel {$g\ge 4$} at 1550 00
\endlabellist
\includegraphics[scale=.27]{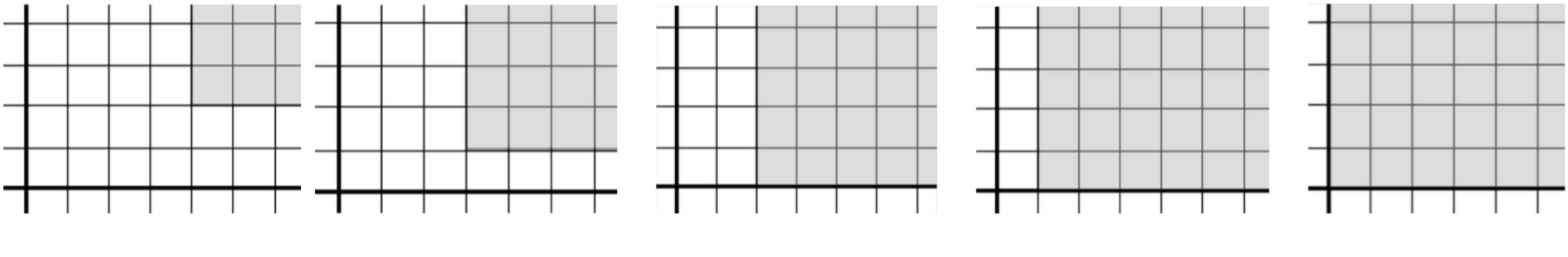} 
\caption{$\calg_g(4P_1, 2P_2)$.}
\label{fig:graphic2}
\end{figure}

\end{example}

\begin{example} Let  $K = 10_{153}$.  We consider cobordisms to the unknot.  For this knot $H_1(M_2(K)) = 0$ and $H_1(M_5(K) ) \cong  \Z_{11} \oplus \Z_{11}$.  Clearly Theorem~\ref{thm:homology-bound} and Corollary~\ref{corollary:bound} provide no information in the case of 2--fold covers.  Using  5--fold covers does.  

We first observe that there are precisely two nontrivial  $5$--eigenspaces in  $ \Z_{11} \oplus \Z_{11}$, each 1--dimensional, as  can be seen as follows.  Clearly there are at most two nontrivial eigenspaces.  Poincar\'e duality implies that if there is a $\zeta$--eigenvector, there is also a $\zeta^{-1}$--eigenvector; we present a proof in this in the  appendix  as Lemma~\ref{thm:eigenequal}. 

Using either eigenvalue, Theorem~\ref{thm:homology-bound} implies that for any cobordism from $nK$ to the unknot, we have $c_0(\Sigma) \ge  n -g$.    Using Corollary~\ref{corollary:bound} yields the weaker result that $c_0(\Sigma) \ge n/2 -g$.   The improvement by a factor of two is expected, since two of the eigenspaces are trivial and two have dimension 1.

\end{example}

\section{The infinite cyclic cover and the Alexander module}
It has been known that the rank of the Alexander module of a knot has an upper bound that is determined by the genus of a surface bounded by the knot in $B^4$ and the critical point structure of that surface.  We now generalize that observation, focusing on cobordisms.  

Recall that  $M_\infty(K_i) $ and  $W_\infty(\Sigma)$ represent the infinite cyclic covers of the complements of the $K_i$ and $\Sigma$.  In general, suppose we  have a finite CW--complex $X$ and a homomorphism  $\rho \co H_1(X) \to \Z$; then $\rho$  induces an infinite cyclic cover $\widetilde{X}_\rho$.    The group $H_1(\widetilde{X}_\rho, \Q)$ is a finitely generated module  over the PID  $\Q[t,t^{-1}]$.  We denote this module by $A(X, \rho, \Q[t,t^{-1}])$.  There is   splitting 
\[A(X, \rho, \Q[t,t^{-1}]) \cong 
\bigoplus_{j=1}^n \Q[t,t^{-1}]/\left< f_j(t)\right>,
\]  where   $f_j$ divides $f_{j+1} $ for $j<n$.   This splitting is unique and the value of $n$ is called the rank of the module.

\begin{definition} Let  $X$ be a space (or pair of spaces) supporting a map $\rho \co H_1(X) \to \Z$ with associated infinite cyclic cover $\widetilde{X}_\rho$. We denote by   $\beta_i(\wX_\rho, \rho, \Q[t,t^{-1}])$    the $\Q[t,t^{-1}]$--rank of $A(X, \rho , \Q[t,t^{-1}])$.   When $\rho$ is implicit, it is dropped from the notation.
\end{definition}

For the  complements of the $K_i$ and of $\Sigma$ there are canonical maps of the first homology to $\Z$, and thus we can suppress the ``$\rho$'' in our notation.  The infinite cyclic cover  $W_\infty(\Sigma)$ is built from the infinite cyclic cover    $M_\infty(K_0) $ by adding the lifts of $c_0$ 1--handles, followed by $c_1$ $2$--handles, and then $3$--handles.  There is a similar decomposition arising for $M_\infty(K_1)$.  The proof of Theorem~\ref{thm:homology-bound} carries over to this setting, yielding the following result.

\begin{theorem} \label{thm:homology-bound3} Suppose that $\Sigma$ is a cobordism from $K_1$ to $K_0$.    Then
\[
c_0(\Sigma) \ge \frac{\beta_1 (M_\infty( K_1) ,\Q[t,t^{-1}])  -\beta_1 (M_\infty(K_0) ,\Q[t,t^{-1}])}{2} -  g(\Sigma).
\]
\end{theorem}

This result can be strengthened by focusing on the direct sum decomposition of the module $A(X, \rho, \Q[t,t^{-1}])$ that corresponds to irreducible elements in  $\Q[t,t^{-1}]$.    For any irreducible polynomial $f$ we can set $\beta_i^f(\wX, \rho, \Q[t,t^{-1}])$ to be the rank of the $f$--primary summand of $A(X, \rho, \Q[t,t^{-1}])$. 
 The proof of the following result is much the same as that for the previous theorem.  (As an alternative, one can switch to the ring $\Q[t,t^{-1}]_{(f)}$, which denotes the localization at $f$, that is, the ring formed from  $\Q[t,t^{-1}]$ by adding a multiplicative inverse to all nontrivial elements $g$ that are relatively prime to $f$.  This is a PID with a unique prime, represented by $f$.)

\begin{theorem} \label{thm:homology-bound2} Suppose that $\Sigma$ is a genus $g$ cobordism from $K_1$ to $K_0$.    Then for any irreducible polynomial $f \in \Q[t,t^{-1}]$, 
\[
c_0(\Sigma) \ge \frac{\beta_1^f (M_\infty( K_1) ,\Q[t,t^{-1}])  -\beta_1^f (M_\infty( K_0) ,\Q[t,t^{-1}])}{2} -  g(\Sigma).
\]

\end{theorem}
The following corollary is immediate.
\begin{corollary} If knots $K$ and $J$ have nontrivial   Alexander polynomials with a pair of distinct irreducible factors, then for any cobordism $\Sigma$ from  from $nK$ to $mJ$ we have 
\[ c_0(\Sigma) \ge n/2 - g 
\]
and
\[ c_2(\Sigma) \ge m/2 - g. 
\]
\end{corollary}

For related results  in the case of ribbon concordances, see~\cite{MR4179701}.

\section{Knots $K(1, \alpha 6_1)$, $K(1, \beta 10_3)$, and their associated  metacycle covers.}\label{sec:lens}

A {\it metacyclic} invariant of a knot $K$, or of a surface $\Sigma \subset S^3 \times [0,1]$,  is one that is derived from a   cyclic cover of the branched cover of $K$ or $\Sigma$.  The use of such invariants in knot theory already appears in early work, such as Reidemeiser's 1932 book~\cites{MR717222,MR0345089}.  The role of such invariants in concordance first appeared   in the work of Casson and Gordon~\cite{MR900252}.  That paper, which introduced what is now called  Casson-Gordon theory,  was restricted to   2--bridge knots $B((2k+1)^2,2)$.  We will build our examples using the   2--bridge knots $B((2k+1)^2,2k)$.  The reason for the different choice is that Casson and Gordon were interested in showing that particular knots are not slice; we want to start with knots that are slice and explore their slice disks and concordances between them.

Our   examples are  built from   two knots from this family:  $K(1, U) = B(9, 2) = 6_1$ and $K(2,U) = B(25,4) = 10_3$, but further examples  are easily constructed.

Figure~\ref{fig:KkJ2} gave an  illustration of a knot $K(k, J)$.  For $J$ unknotted, this is $B((2k+1)^2, 2k)$.  
We can think of   $K(1,J)$ as being built from $B((2k+1)^2, 2k)$ by removing a neighborhood of a circle linking the right band in the Seifert surface shown in Figure~\ref{fig:KkJ2} (for which  the right band unknotted) and   replacing that  neighborhood    with the complement of the  knot $J$ in $S^3$.   The identification of the boundaries interchanges the meridian and longitude.  This creates a new knot in $S^3$, formed from $K_1(U)$ by tying the knot $J$ in a band on the Seifert surface, as desired.    We will focus on two specific examples: $K(1,\alpha 6_1)$ and $K(1,\beta 10_3)$, where $\alpha$ and $\beta$ are nonnegative integers.

\subsection{Ribbon disks for $K(k, J)$}

\begin{theorem}  If $J$ is ribbon and bounds a ribbon disk with $n$ minima, then $K(k,J)$ is ribbon, bounding a ribbon disk with $2n$ minima.
\end{theorem}
\begin{proof}
The knot $B((2k+1)^2, 2k)$ is ribbon: a simple closed curve $\gamma$ that passes over both bands of the Seifert surface once  is unknotted and has framing zero.    The ribbon disk has one index one  critical point  and two minima.  A ribbon disk for $K(1,J)$ is built by removing an annular neighborhood of $\gamma$ (on the Seifert surface for $K(k,J)$) and replacing it with a pair of ribbon disks for $J$.  
\end{proof}

\subsection{The 2--fold branched cover of $K(1,J)$}  An algorithm of Akbulut-Kirby~\cite{MR593626} provides a surgery diagram of the 2--fold branched cover of $K(1,J)$, as shown   on the left  in Figure~\ref{fig:meta1}; $M_2(K(1,J))$ is given as surgery on a two-component link, with one of the components unknotted and the other representing $J \cs J^r$, where $J^r$ denotes $J$ with its string orientation reversed.  Since all the knots $J$ we consider are reversible, we have left out the superscript ``$r$'' and do not orient the circles labeled with $J$.  Also, we can write $2J$ rather then $J \cs J^r$ when needed.  As describe in, for instance, ~\cite{MR0515288}, that surgery diagram can be modified to appears the     diagram on the right.  This  illustrates the 2--fold branched cover as formed from the lens space $L(9,2)$ by removing two parallel copies of a core circle and replacing each with a copy of the complement of $J$.  

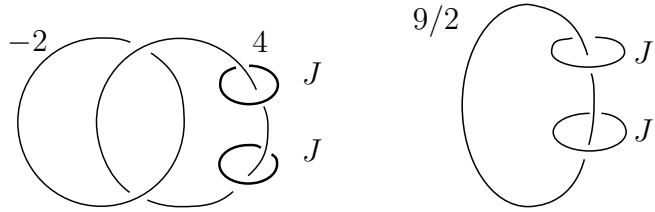
\begin{figure}[h]
\labellist
 \pinlabel {\text{{$-2$}}} at 10 160
  \pinlabel {\text{ {$4$}}} at 230 160
    \pinlabel {\text{ {$J$}}} at 600 150
        \pinlabel {\text{ {$J$}}} at 600 75
          \pinlabel {\text{ {$J$}}} at 280 130
        \pinlabel {\text{ {$J$}}} at 280 55
              \pinlabel {\text{ {$9/2$}}} at 400 180
\endlabellist
\includegraphics[scale=.4]{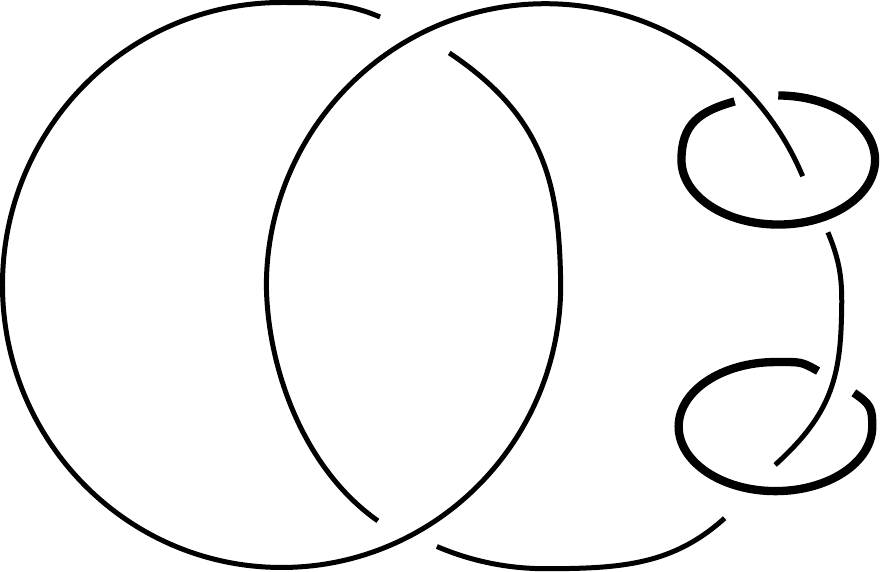} \hskip.9in \includegraphics[scale=.4]{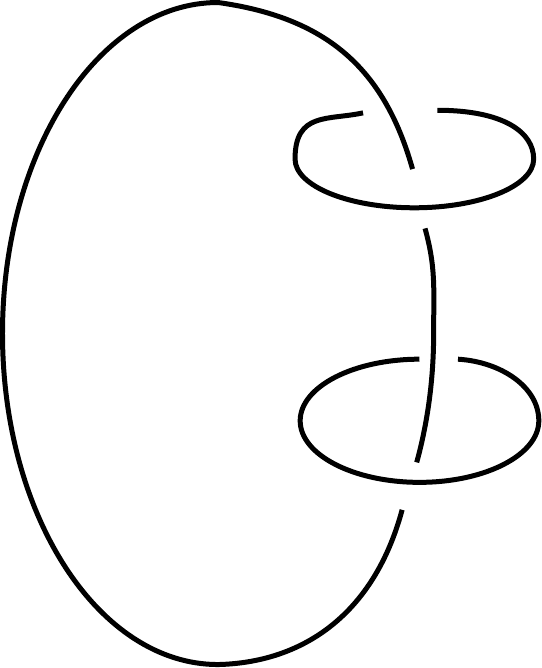} 
\caption{Branched cyclic cover of $K(1,J)$, where $J = J^r$.}
\label{fig:meta1}
\end{figure}

  As an immediate consequence,  we have the following.

\begin{theorem}  For all $J$, $H_1(M_2(K(1,J))) \cong \Z_9$.  \end{theorem}

\subsection{The homology of the  metacyclic cover of $K(1,J)$} 

It is evident that the  3--fold cyclic cover of $M_2(K(1,J))$ is built from the 3--fold cyclic cover of $L(9,2)$, which is the lens space $L(3,2)$,  by removing a pair of parallel core circles and replacing them with copies of $M_3(J) \setminus \widetilde{J}$.  This is illustrated in Figure~\ref{fig:meta2}.    We will thus need the following.    Let $\wM_2^3(K(1,J))$ denote the nontrivial  3--fold cyclic cover of $M_2(K(1,J))$.  

\begin{figure}[h]
\labellist
    \pinlabel {\text{ {$\widetilde{J}$}}} at 175 150
        \pinlabel {\text{ {$\widetilde{J}$}}} at 175 75
              \pinlabel {\text{ {$3/2$}}} at -08 170
\endlabellist
  \includegraphics[scale=.5]{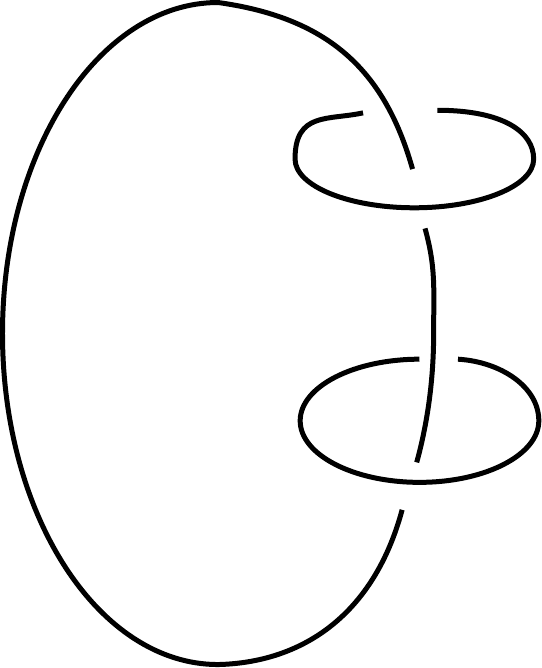} 
\caption{The $3$--fold cyclic cover of the 2--fold branched cover of $K(1,J)$.}
\label{fig:meta2}
\end{figure}

\begin{theorem}   There is an isomorphism $H_1(  \wM_2^3(K(1,J))) \cong \Z_{3} \oplus H_1(M_3(J))^2$. 
\end{theorem}

\begin{proof}  For any knot $J$, let $X_1$ and $X_2$ be copies of the $3$--fold cyclic cover of $S^3 \setminus J$.  We have $H_1(X_i) \cong \Z \oplus H_1(M_{3}(J) )$.  

The torus boundary of $X_i$ has natural boundary  curves, $m_i$, and $l_i$, lifts of the meridian and longitude of $J$.   The curve $m_i$ represents an element of infinite order in  $\Z \oplus H_1(M_3(J) )$,  and   after a change of basis represents $1 \oplus 0$. The curve $l_i$ is null-homologous in $X_i$, bounding a lift of a Seifert surface. 

In Figure~\ref{fig:meta2}   the curves $m_i$ and $l_i$ are attached to the longitude and meridian, respectively, of the curves labeled $\widetilde{J}_3$. (Notice that there is an interchange of meridian an longitude.)

One can now undertake a   Mayer-Vietoris argument.  The covering space is   split into four components by the  three evident tori in Figure~\ref{fig:meta2}, that is, the peripheral tori to the three curves illustrated.  As just described, two are related to the $3$--fold  covers of $J$, one is a solid torus with core $\gamma$ (corresponding to the $3/2$--surgery), and one is the compliment of the three component link that is illustrated, having homology generated by three meridians, which we denote $\alpha_0$, $\beta_1$ and $\beta_2$, corresponding to the $3/2$--surgery curves and the two $\widetilde{J}_{3}$.  We let $ T =  H_1(M_{3}(J) ) $.   Via the Mayer-Vietoris sequence, we see the homology is a quotient of 
\[
\big( \Z(\alpha) \oplus \Z(\beta_1) \oplus \Z(\beta_2)   \big) \oplus  \big( \Z(m_1) \oplus T \big) \oplus   \big( \Z(m_2) \oplus T\big) \oplus \Z(\gamma).
\]

The identification along the three tori, each with rank two first homology, introduces six relations.  Taking them in order, meridian first and initially along the surgery torus, yields the following, where we write $l_1$ despite it equaling 0, to make the gluing maps more evident:

\begin{itemize}
\item $\alpha = -2 \gamma$
\item $\beta_1 + \beta_2 = 3\gamma$ 
\item $\beta_1 = l_1$
\item $\alpha = m_1$
\item $\beta_2 = l_2$
\item $\alpha = m_2$.
\end{itemize}
None of these involve the summand $T \oplus T$ and so,   in effect,  they are relations defining a quotient of $\Z^6\cong  \left< \alpha, \beta_1, \beta_2, m_1, m_2, \gamma\right> $.  A simple exercise shows the quotient is isomorphic to  $\Z_{3}$, as desired.   In our case, we have either $T \cong (\Z_{7})^2 $ or $T \cong (\Z_{19})^2 $.
\end{proof}

To apply this result, we will use the following and its immediate corollary.

\begin{lemma}\label{lem:homcover}   $H_1(M_3(6_1)) \cong \Z_7 \oplus \Z_7$ and  $H_1(M_3(10_3)) \cong \Z_{19}      \oplus \Z_{19}$.\end{lemma}

\begin{proof}  This is a standard knot theoretic computation; see, for instance~\cite{MR0515288}. More generally, as described in the appendix, one can readily show that for $q$ odd,  $H_1(M_q(B((2k+1)^2, 2k))) \cong \Z_{(k+1)^q - k^q} \oplus  \Z_{(k+1)^q - k^q} $.
\end{proof}

\begin{corollary} $H_1(\wM_2^3(K(1, 6_1)) \cong \Z^3 \oplus (\Z_7)^4$ and $H_1(\wM_2^3(K(1,10_3)) \cong \Z^3 \oplus (\Z_{19})^4$. 

\end{corollary}

\subsection{The eigenvalue decomposition of $H_1(\wM_2^3(K(1, J)))$.} 

For any field $\F$, there  is an action of $\Z_3$ on $H_1(\wM_2^3(K(1, J)), \F)$.  In the case that $\F$ contains a primitive $3$--root of unity $\zeta$, the homology  $H_1(\wM(K(1, J)), \F)$ splits into eigenspaces, as described in Section~\ref{sec:equivariant}.  Note that $\F_7$ and $\F_{19}$ both contain such roots of unity.  When no confusion can result, we will use the same symbol $\zeta$ to denote a primitive  cube roots of unity in $\F_7$ and in $\F_{19}$.    

\begin{theorem} With the set-up described above:
\begin{itemize}
\item  $\beta_1^\zeta( \wM_2^3(K(1, \alpha 6_1)), \F_7) = 2\alpha$.
\item  $\beta_1^\zeta( \wM_2^3(K( 1, \alpha 6_1)), \F_{19}) =0$.
\item  $\beta_1^\zeta( \wM_2^3(K(1,  \beta 10_3)), \F_7) = 0$.
\item $ \beta_1^\zeta( \wM^3_2(K(1,  \beta 10_3)), \F_{19}) =2 \beta$.
\end{itemize}

\end{theorem}

\begin{proof} Considering the $\F_7$--homology, we have   $H_1( \wM_2^3(K(1, \alpha 6_1)), \F_7) \cong (\F_7)^{2\alpha}$ arises entirely from the $2\alpha $ copies of $M_3(J) \setminus \widetilde{J}$ that appear in the covering space.   Thus the proof of the first statement comes down to analyzing the eigenspace splitting of the $Z_3$--action on $M_3(J) \cong \F_7 \oplus \F_7$.  We claim that the $1$--eigenspace is trivial and the $\zeta$--eigenspaces  and $\zeta^{-1}$--eigenspaces are 
both 1--dimensional.   This can be shown with an explicit computation, or one can argue abstractly, as follows.  A transfer argument, using the branched covering map $M_3(J) \to S^3$ shows that the 1--eigenspace is trivial.  Poincar\'e duality implies that the $\zeta$--eigenspace and $\zeta^{-1}$--eigenspace are isomorphic (see Lemma~\ref{thm:eigenequal} for a proof).

Similar arguments give the remaining statements.
\end{proof}

\subsection{Metacyclic covers of $nK(1, \alpha 6_1)$ and  $mK(1, \beta 10_3)$}

Let $\rho \co H_1(M(nK(1,J)) ) \to \Z_3$ be nonzero on $a$ of the natural $\Z_3$--summands and be $0$ on $(n-a)$ of the summands. We wish to understand the eigenspace decomposition of the homology of the associated cover.  This will be clarified by the following result concerning the lens space $L(9,2)$. It can be  proved with a simple construction and  should make the subsequent results evident.

\begin{lemma} Let $X = L(9,2)$.  Suppose that $\rho \co H_1(nX) \to \Z_3$ is nonzero on $a \ge 1$ of the natural $\Z_9$--summands.  Then the associated $3$--fold cover satisfies
\[
 \widetilde{nX}_3\cong aL(3,2) \cs 3(n-a) L(9,2) \cs 2(a-1) S^1 \times S^2.
\]
\end{lemma}

\begin{theorem}$\ $

\noindent {\bf A.} Suppose that $\rho \co H_1(nM_2(K(1,\alpha 6_1))) \to \Z_3$ is nonzero on $a\ge 0$ of the natural $\Z_9$--summands.  Then

\begin{itemize}
\item   If $a \ge 1$, then $\beta_1^\zeta( \wM_2^3(nK(1, \alpha 6_1)), \rho, \F_7) = 2a\alpha +  a-1 $.
\item    If $a \ge 1$, $\beta_1^\zeta( \wM_2^3(nK(1, \alpha 6_1)), \rho, \F_{19}) =   a-1 $.
\item If $a=0$, then $\beta_1^\zeta( \wM_2^3(nK(1, \alpha 6_1)), \rho, \F_7) =  \beta_1^\zeta( M_2(nK(1, \alpha 6_1)), \rho, \F_{19}) =   0 $.
\end{itemize}

Similarly, 

\noindent {\bf B.} Suppose that $\rho \co H_1(nM_2(K(1,\beta 10_3)) \to \Z_3$ is nonzero on $a'\ge 0$ of the natural $\Z_9$--summands.  Then

\begin{itemize}
\item   If $a' \ge 1$, then $\beta_1^\zeta( \wM_2^3(nK(1, \beta 10_3)), \rho, \F_{19}) = 2a'\beta +  a'-1 $.
\item    If $a' \ge 1$, $\beta_1^\zeta( \wM^3_2(nK(1, \beta 10_3)), \rho, \F_{7}) =   a'-1 $.
\item If $a'=0$, then $\beta_1^\zeta( \wM^3_2(nK(1, \beta 10_3)), \rho, \F_{19}) =  \beta_1^\zeta( M_2(nK(1, \beta 10_3)), \rho, \F_{7}) =   0 $.
\end{itemize}

\end{theorem}

\begin{proof}  Most of the terms that appear in the statements are evident from the construction, with perhaps one exception.  In the first formula there is the term $a-1$ which arises from the $S^1 \times S^2$ summands.  To clarify this, we will consider the case of $a=2$ and the more general situation of $L(mn, q)$ with the homomorphism $\rho$ mapping onto $\Z_n$ on both factors.  Then the $n$--fold cyclic cover is $L(m,q) \cs L(m,q) \cs (n-1) S^1 \times S^2$.  The homology with $\F$ coefficients has a summand $\F^{n-1}$.   As a $\Z_n$--module this is $\F[\Z_n]/\left< 1-t \right>$.  In the case that $\F$ contains a primitive $n$--root of unity, this splits into $(n-1)$   eigenspaces of dimension 1. 
\end{proof}


\section{Cobordisms between  $nK(1, \alpha 6_1)$ and $mK(1,\beta 10_3))$.}\label{sec:6_1}

To simplify the discussion, we will assume that  $n \ge m > 0$.  Let  $\Sigma$ be a genus $g$ cobordism  from  $nK(1,\alpha 6_1)$ and $mK(1,\beta 10_3))$.  We continue to denote  the 2--fold  cover of $S^3 \times [0,1]$ branched over $\Sigma$  by $W_2(\Sigma)$; this is a cobordism from   $M_2(nK(1, \alpha 6_1))$ to $M_2(mK(1, \beta 10_3))$.  

\subsection{Extending homomorphisms from $H_1(M_2(nK(1, \alpha 6_1)))$ to $H_1(W_2(\Sigma))$.}

 We have that the  $M_2(nK(1, \alpha 6_1))$ and $M_2(nK(1,\beta 10_3))$ are $\Q$--homology spheres and thus there are $\Q/\Z$--valued non-singular symmetric  linking forms on  each one.  For any $\Q$--homology sphere $M$, the linking form provides  an  identification of   $H_1(M)$ with $\hom (H_1(M), \Q/\Z)$.  We remind the reader that a {\it metabolizer}  for such a linking form on an abelian group of order $l^2$ is a subgroup of order $l$ on which the linking form is identically 0.  

The following result is an immediate consequence of a theorem of Gilmer in~\cite{MR656619}.  In summary, suppose that $K$ bounds a genus $g$ surface in $B^4$.  Then according to~\cite[Lemma 1]{MR656619}, the homology group $H_1(M_2(K))$  splits as a direct sum $A \oplus B$, where $B$ has a presentation of size $2g \times 2g$ and the linking form on $A$ is metabolic; notice that this implies that if $B \cong \Z_n^k$, then $k \le 2g$.  
Denote by $\phi$ the restriction map  
\[
\phi \co \hom  \big(H_1(W_2(\Sigma)), \Q/\Z\big) \to \hom \big(H_1(nM_2(K_1)) \oplus H_1(-M_2(mM_2(K_0)), \Q/\Z\big).
\] 

\begin{theorem} \label{thm:gilmer3}Suppose that $n+m \ge 2g$ and recall there is an    isomorphism $H_1( nM_2(K_1))  \oplus H_1(-mM_2( K_0)) \cong (\Z_9)^{n} \oplus  (\Z_9)^{m} $.   For some $\epsilon \ge 0$, the linking form on this group splits off a summand    that is isomorphic to  $(\Z_{9})^{n + m - 2g + \epsilon}$ which  contains a metabolizer $\calm \subset H$, all elements of which    are in the image of $\phi$.  In particular, the order of $\calm$ is at least $3^{n+m-2g}$.
\end{theorem}

To apply this result, we clearly need to have $n +  m > 2g$.  To simplify our considerations, we will assume that $n>2g$. 

\begin{corollary} Suppose that $\Sigma$ is a genus $g$ cobordism from $nK(1, \alpha 6_1)$ to $mK(1, \beta 10_3)$ and assume that  $n > 2g$.  Then there is a surjective homomorphism $\rho  \co H_1(M_2(nK(1, 6_1))) \to \Z_3$ that extends to a homomorphism $ {\rho'}\co H_1(W_2(\Sigma)) \to  \Z_{3^b}$ for some $b$.

\end{corollary}

\begin{proof} Theorem~\ref{thm:gilmer3} provides a set $\calm$ of homomorphisms   
$   H_1(nM_2(K_1)) \oplus H_1(-M_2(mK_0) )  \to \Q/\Z$
that extend  to   homomorphisms $\rho'$ on $H_1(W_2(\Sigma) )$.   The order of $\calm$ is $3^{n+m -2g}$ and the  order of a metabolizer for $  H_1(-M_2(mK_0))$ is $3^{m}$.  It follows that if $3^{n+m -2g} > 3^m$, then some element in $\calm$ is not contained in  $0  \oplus H_1(-M_2(mK_0))$ and thus must be nontrivial on  $ H_1( M_2(mK_1))$.  This will occur as long as  $n > 2g$.  Call one such element $\rho$ and let $\rho'$ denote an extension of $\rho$  to $H_1(W_2(\Sigma) )$.   

The image of $\rho'$ is a finite cyclic subgroup $G \subset \Q/\Z$.  Projecting $G$ to its $3$--primary summand does not change its restriction to the boundary, so we can assume that  $\rho'$ takes values in $\Z_{3^b}$ for some $b$.    If $\rho$ is not of order $3$, then it can be multipled by $3$ so that it does have order 3.
\end{proof}

\subsection{The $3^b$--fold cyclic cover of $W(\Sigma)$.}

Let $\pi\co \wW_2^3(\Sigma) \to W_2(\Sigma)$ denote the $3^b$--fold cyclic cover of $W_2(\Sigma)$ associated to the homomorphism $\rho'$ defined above.   We let $\partial_1(\wW_2^3) = \pi^{-1}(  M_2(nK(1, \alpha 6_1))) $ and $\partial_0(\wW_2^3) = \pi^{-1}(  M_2(nK(1, \alpha 10_3))) $.  

We can now apply Theorem~\ref{thm:equivcount}.  Let $\zeta$ be a primitive  3--root of unity and consider the $\Z_3$--action on $\partial_1(\wW_2^3)$, the $3^{b-1}$ power of order $3^b$ deck transformation, which we denote by $T = S^{3^{b-1}}$.  

\begin{theorem}\label{thm:cover2}   Assume that the restriction $\rho \co M_2(nK(1, \alpha 6_1)) \to \Z_3$ is nonzero on $a \ge 1$ of the $n$ summands.  Also suppose that the restriction is nonzero on $a' \ge 0$ of the $m$ summands of $H_1(M_2(mK(\beta 10_3))$.

\begin{itemize}
\item  $\beta_1^\zeta( \partial_1(\wW_2^3),T,  \F_7 )= 3^{b-1} \beta_1^\zeta(\wM_2^3(nK(1,\alpha 6_1)), \F_7) = 3^{b-1} (2a \alpha +a -1) $.

\item   $\beta_1^\zeta( \partial_0(\wW_2^3),T,  \F_7 )= 3^{b-1} ( a' -1) $ if $a' \ge 1$ and $\beta_1^\zeta( \partial_0(\wW_2^3), \F_7 )= 0$ if $a'=0$.
\end{itemize}

\end{theorem}

Applying a relative version of Theorem~\ref{thm:cocount} along with Theorem~\ref{thm:handles} gives the next result.

\begin{theorem} Let  $C_i^\zeta(\wW_2^3, \partial_0 (\wW_2^3), \F_7)$ be the $\zeta$--eigenspace of the CW--chain group under the $\Z_3$--action given as the $3^{b-1}$--power of its deck transformation.  Then

\begin{itemize}

\item $ \dim C_1^\zeta(\wW_2^3, \partial_0 (\wW_2^3),T,  \F_7) = 3^{b-1} (2c_0(\Sigma))$.

\item $ \dim C_2^\zeta(\wW_2^3, \partial_0 (\wW_2^3), T, \F_7) = 3^{b-1} (2c_1(\Sigma))$.

\item $ \dim C_3^\zeta(\wW_2^3, \partial_0( \wW_2^3),T , \F_7) = 3^{b-1} (2c_2(\Sigma) + 2g(\Sigma))$.

\end{itemize}

\end{theorem}

\begin{theorem}\label{thm:metbound} Let $\Sigma$ be a genus $g$ cobordism from $nK(1, \alpha 6_1)$ to $mK(1, \beta 10_3)$.  Assume that $n > 2g$.  Then 
\[ 
c_0(\Sigma) \ge \frac{2\alpha +1 -m}{4} -g.
\]

\end{theorem}

\begin{proof} The proof is much like the one for Theorem~\ref{thm:homology-bound}.   We work with the $\zeta$--eigenspaces of the $\Z_3$--actions.   Consider the fact that $\wW_2^3$ is built from $\partial_0( \wW_2^3)$.   We have  

\[  
\beta_1^\zeta (\wW_2^3 ,T,  \F_7)  \le 3^{b-1} (m-1) + 3^{b-1} (2c_0(\Sigma)).
\]
The first summand comes from the homology of the boundary, using the fact that in Theorem~\ref{thm:cover2} we have $a'-1 \le m$.  Turning the bordism upside down and using that fact that  $\wW_2^3$ is built from $\partial_1 (\wW_2^3)$ by adding 1--handles and 2--handles that correspond to the index two and index one critical points of $\Sigma$, respectively,  we find that 
\[  
\beta_1^\zeta (\wW_2^3, T, \F_7) \ge  3^{b-1} ( 2a \alpha +a -1  )+ 3^{b-1} (2c_2 (\Sigma))- 3^{b-1} (2c_1(\Sigma)).
\] 

Together, these inequalities imply  
\[  
( 2a \alpha +a -1  )+    2c_2 (\Sigma) -   (2c_1(\Sigma))  \le   (m-1) +   (2c_0(\Sigma)).
\] 
We have that $c_1(\Sigma) = c_0(\Sigma) + c_2(\Sigma) + 2g(\Sigma)$.  Substituting yields  
\[
( 2a \alpha +a -1  )+   2c_2 (\Sigma) -   2(    c_0(\Sigma) + c_2(\Sigma) + 2g(\Sigma)          )\le   (m-1) +    2c_0(\Sigma).
\] 
This simplifies to give 
\[ 
c_0(\Sigma) \ge \frac{2a\alpha +a -m}{4} -g.
\]
Finally, since $ a\ge 1$, we have the desired result: 
\[ 
c_0(\Sigma) \ge \frac{2\alpha +1 -m}{4} -g.
\]  
\end{proof}

\subsection{Strengthening the bounds}

The difference between the lower bound provided by Theorem~\ref{thm:metbound} and the best upper bound that we can prove with a realization  result   is quite large. For instance, we have the following realization result.

\begin{theorem} If $g \le  \min\{ n(2\alpha +1) ,  m(2\beta +1)\}$, then there is a genus $g$ cobordism $\Sigma$  from $nK(\alpha 6_1)$ to $mK(\beta 10_3)$ satisfying 
\[ c_0(\Sigma) = n(2\alpha +1) -g \hskip.4in \text{and} \hskip.4in  c_0(\Sigma) = m(2\beta +1) -g  
\]

\end{theorem}

\begin{proof}
The construction given in Example~\ref{ex:genus-grow} can be easily modified to produce the result.  What is essential is that the canonical ribbon disks can be pieced together to form a concordance in  which the local maxima are beneath the local  minima.  
\end{proof}

A limitation in this theorem is the  absence of $n$ in the bound on $c_0$ given Theorem~\ref{thm:metbound}.  We want to explore this briefly.  We have an  inclusion of $(\Z_3)^{m+n} $ into a group with nonsingular linking form:
\[ (\Z_3)^n \oplus (\Z_3)^m \subset (\Z_9)^n \oplus (\Z_9)^m.\]  
We have assumed that $n +m  > 2g$ and identified a   subgroup $\calm \subset (\Z_9)^n \oplus (\Z_9)^m$ of order $3^{n+m - 2g}$ upon which the linking form is identically 0.   In the proof of Theorem~\ref{thm:metbound}, we used the fact that if $n >2g$ then $\calm \cap (\Z_3)^n \oplus 0$ is nontrivial.  But in fact, if $n$ is large in comparison to $m$ and $g$, then the rank of the intersection  $\calm \cap (\Z_3)^n \oplus 0$ must be large as well; in particular, rather than use $a\ge 1$ in the argument, we could find metabolizing elements for which $a$ is much larger.  Similar, we used the obvious fact that $a' \le m$; with care, we could also show that it is possible to assume that $a'$ is close to 0.  We have opted not to undertake the careful analysis of self-annihilating subgroups of the standard linking form on $(\Z_9)^{n+m}$ that  is required to establish these better bounds. 


\section{Non-reversible knots}

To conclude our presentation of examples, we consider a subtle family of examples built from knots $K$ and $K^r$, where $K^r$ denotes the reverse of $K$.  Such knots are difficult to distinguish by any means.  For instance, all abelian invariants are identical for the two knots.  It is not known at the moment whether any invariants that are built from the Heegaard Floer knot complex $\cfk^\infty(K)$ defined in~\cite{MR2065507}, such as  its involutive counterpart, defined in~\cite{MR3649355}, can distinguish them.  The successful application   of metacyclic invariants to distinguishing knots from their reverses began with the work of Hartley,~\cite{MR683753}.

 Figure~\ref{fig:reversible}  illustrates a knot that we will denote $P = P(J_1,J_2)$.  The starting knot is the pretzel knot $P(3,-3,3)$, and knots $J_1$ and $J_2$ are placed in the two bands.  Notice that we have indicated the orientation of $P$.  We let $P^*$ denote reverse of the knot; that is, the knot with the same diagram except with the arrow reversed (the  use of  $P^*$ rather than the more standard notation  $P^r$ will simplify some notation later on).  These knots have formed the basis of a variety of concordance result related to reversibility; see, for instance,~\cite{2019arXiv190412014K}.  
 In past papers that used  these knots, the $J_1$ were chosen so that the knots could be shown not to be concordant.  We will let $J_1$ and $J_2$ be slice knots, so that they the $P$ and $P^*$ are themselves slice, and or results apply to consider concordances between them.

\begin{figure}[h]
\labellist
 \pinlabel {\text{\large{$J_1$}}} at 20 360
 \pinlabel {\text{\large{$J_2$}}} at 660 360
\endlabellist
\includegraphics[scale=.2]{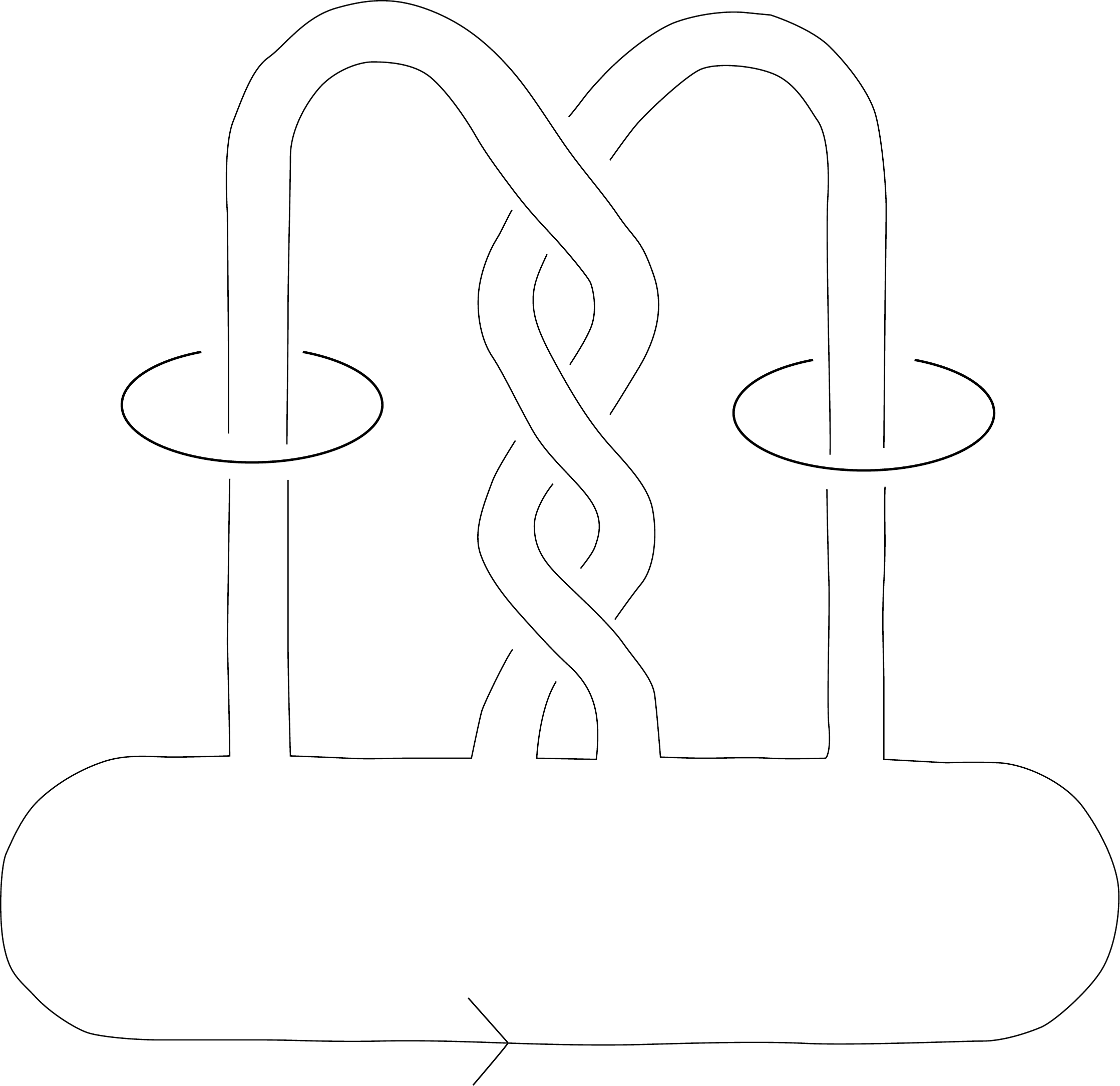} 
\caption{The knot $P(J_1, J_2)$.}
\label{fig:reversible}
\end{figure}

We will now briefly summarize the results of some calculations related to these, leaving the   details   to~\cite{2019arXiv190412014K}.  First, we have for the 3--fold cover that $H_1(M_3(P)) \cong \Z_7 \oplus \Z_7$.    This group splits into a 2--eigenspace and a 4--eigenspace for the deck transformation using $\F_7$--coefficients.  If $z$ and $w$ are linking circles to the two bands, with $\widetilde{z}$ and $\widetilde{w}$ being chosen lifts to $M_3(P)$, then the 2--eigenspace and 4--eigenspace are spanned by  $\widetilde{z}$ and $\widetilde{w}$, respectively.  For $P^*$ we have $\widetilde{z}^*$ and $\widetilde{w}^*$ as  eigenvectors, but because of the reversal, they are the 4--eigenvectors and 2--eigenvectors, respectively.

Let $\Sigma$ be a cobordism from $P$ to $P^*$.  

We let  $W_3$ be the 3--fold   cover of $S^3 \times [0,1]$ branched over $\Sigma$. In Section~\ref{sec:6_1} we considered a metabolizer $\calm$  of the linking form on $H_1(M_3(P)) \oplus H_1(-M_3(P^*))$.  This metabolizer must be invariant under the $Z_3$--action, and thus is spanned by eigenvectors.  Here are the possibilities.

\begin{itemize}
\item $\calm$ is a 2--eigenspace, spanned by $\{\widetilde{z} , \widetilde{w}^*\}$.

\item $\calm$ is a 4--eigenspace, spanned by $\{\widetilde{w} , \widetilde{z}^*\}$.

\item $\calm$ contains a nontrivial   2--eigenvector $a \widetilde{z }+ b \widetilde{w}^*$  and a 4--eigenvector $c \widetilde{w }+ b \widetilde{z}^*$.

\end{itemize}

We now wish to find obstructions based on the $7$--fold cyclic covers of the spaces involved.  There are three cases to consider.  Here is a summary of what arises.

\begin{itemize}
\item {\bf Case 1:}  Considering the eigenvector $\widetilde{z} $, the corresponding cover of $M_3(P)$ will have first homology that depends on the homology of $M_7(J_1)$. For the  eigenvector $\widetilde{w}^*$, the corresponding cover of $M_3(P^*)$ will have first homology that depends on the homology of $M_7(J_2)$.   

\item {\bf Case 2:} This is similar.  For the eigenvector $\widetilde{w} $, the corresponding cover of $M_3(P)$ will have first homology that depends on the homology of $M_7(J_2)$. For the  eigenvector $\widetilde{z}^* $, the corresponding cover of $M_3(P^*)$ will have first homology that depends on the homology of $M_7(J_1)$. 

\item {\bf Case 3:} The last case splits into subcases, depending on whether the coefficients $a$, $b$, $c$, and $d$ are zero or not.  The most interesting case is when, say $a \ne 0 \ne b$. Then the corresponding 7--fold cover of $M_3(P)$ will involve the first homology of $M_7(J_1)$ and the corresponding 7--fold cover of $M_3(P^*)$ will also involve the first homology of $M_7(J_2)$.  

\end{itemize}

From this it should be clear that by choosing $J_1$ and $J_2$ so that the rank of the first homology groups  $H_1(M_7(J_1),\F_p)$ and $H_1(M_7(J_2), \F_{p'})$ are large for appropriate primes $p$ and $p'$,   then regardless of which metabolizer arises, there will be obstructions to the values of $c_0(\Sigma)$ and $c_2(\Sigma)$ being small.  This can be achieved by letting $J_1$ be a multiple of $6_1$ and letting $J_2$ be a multiple of $10_3$.  A computation as described in the appendix shows $H_1(M_7(6_1)) \cong \Z_{127} \oplus \Z_{127}$ and  $H_1(M_7(10_3)) \cong \Z_{2059} \oplus \Z_{2059}$.  The number 2059 has prime factors 29 and 71.  All of $\F_{127}$, $\F_{29}$ and $\F_{71}$ contain primitive $7$--roots of unity.

To construct examples in Section~\ref{sec:lens}, we used the fact that the 3--fold cover of $L(9,2)$ is $L(3,2)$.   For carrying out an explicit computation here, we would need to know the homology of the  7--fold cover of $M_3(P)$ corresponding to each eigenspace of the $\Z_3$--action.  Regardless of what there groups are,  their ranks  in comparison to the rank of $H_1(M_7(\alpha J_1))$ or $H_1(M_7(\beta J_2))$ will be small if  $\alpha $ and $\beta$ are large.  This permits one to  prove the following result.

\begin{theorem}  For any non-negative integers $g$, $c_0$ and $c_2$, there are positive integers $\alpha$ and $\beta$ such that the knot $P = P(\alpha 6_1, \beta 10_3)$ has the following properties.

\begin{itemize}
\item $P(\alpha 6_1, b10_3)$ is a ribbon knot.

\item Any genus $g$ cobordism $\Sigma$ from $P$ to $P^*$ has $c_0(\Sigma) \ge c_0$ and $c_2(\Sigma) \ge c_2$.

\end{itemize}

\end{theorem}


\section{Problems}\label{sec:problems}
\begin{enumerate}

\item  Is $\calg_g(K_1,K_0)$ always a quadrant,    of the form $\calq(a,b)$, for some $a$ and $b$?

\item An affirmative answer to the previous question would be implied by a positive answer to the following:  If $(a+1, b) \in \calg_g(K_1,K_0)$ and $(a, b+1) \in \calg_g(K_1,K_0)$, then is $(a,b) \in   \calg_g(K_1,K_0)$?  This would also imply Gordon's Conjecture~\cite{MR634459}: If $K_1$ is ribbon concordant to $K_0$ and $K_0$ is ribbon concordant to $K_1$, then $K_1 = K_0$.

\item An even simpler generalization of Gordon's conjecture is the following statement:  if for some $c_0$ and $c_2$,  $(c_0,0) \in \calg_g(K_1,K_0)$ and $(0,c_2) \in \calg_g(K_1,K_0)$, then $\calg_g(K_1,K_2) = \calq(0,0)$.

\item If $(a+1,b+1) \in \calg_g(K_1,K_0)$, then is $(a,b) \in \calg_{g+1}(K_1,K_0)$?

\item  \label{prob1} Recall that the bridge number of $K$ is denoted $\text{br}(K)$ and we defined $\text{b}(K)$ to be the minimum number of index 0 critical points of a slice disk for $K \cs -K$.   It is elementary to show that $ \text{b}(K) \le  \text{br}(K)$.  It is also not difficult to  construct ribbon knots $K$ with large bridge index that bound disks in the four-ball with one saddle point.   Using these  knots  we see that $\text{br}(K) - \text{b}(K)$ can be arbitrarily large.  

For the torus knot $T_{2,3}$ we have $\text{br}(T_{2,3}) = 2$ and it is elementary to see that $\text{b}(T_{2,3}) = 2$.  In fact, in~\cite{MR4186142} it is shown that for torus knots,    $\text{b}(K) = \text{br}(K)$.  Yet there are still basic examples that are unresolved: for   $K = nT_{2,3}$ we have  $\text{br}(K) = n+1$; is it true that  $\text{b}(nT_{2,3})=  n+1$?
\end{enumerate}


\appendix

\section{The knots $K(k, J)$}
Here we summarize the computations required in Section~\ref{sec:lens} that determine the homology groups of covering spaces associated to   $K(k,J)$. Recall that if  $J$ is unknotted, this is the two-bridge knot $B((2k+1)^2, 2k)$.   It is the basic building block for the examples in Lemma~\ref{lem:homcover}. 

\subsection{A Seifert surface for  $K(k,J)$ and its Seifert form.} 

The knot $K(k,J)$ has a genus $1$ Seifert surface  $F$ built by attaching two bands to a disk, one with framing $k+1$ and other with framing $-k$.  The first band has a knot $J$ tied it it.  This was illustrated in Figure~\ref{fig:meta1}.  The   Seifert matrix with respect to the natural basis $\{a, b\}$ of $H_1(F)$ is  
\[ 
A_k =   \begin{pmatrix}k+1 & 1 \\
0 & -k \\
\end{pmatrix}.
\]

The classes $a$ and $b$ are represented by simple closed curves on $F$ representing the unknot and the knot $J$.  If we change    basis, letting $a' = a -b$ and $b' = b$ then the Seifert matrix becomes 
\[ 
B_k =   \begin{pmatrix} 0 & k+1 \\
k & -k \\
\end{pmatrix}.
\]
These generators are still represented by simple closed curves, the first of which is unknotted and the second of which represents $J$. 

\subsection{The homology of the cyclic branched covers of $K(k,J)$}  We next have the computation of the needed homology groups.

\begin{theorem}  Let $K(k,J)$ be as above.  Then $H_1(M_2(K(k,J))) \cong \Z_{(2k+1)^2}$.  For $n$ odd, $H_1(M_n(K(k,J))) \cong \Z_{d} \oplus \Z_d$, where $d = (k+1)^n - k^n$.

\end{theorem}

\begin{proof}  The homology group   $H_1(M_2(K(k,J)))$ is presented by $A_k + A_k^{\sf T}$, where $A^{\sf T}$ denotes the transpose.  This $2 \times 2$ matrix has one if its entries a 1, so it presents a cyclic group.  The order of that group is the absolute value of the determinant of the matrix.   As an alternative, the presence of $J$ does not affect the Seifert matrix or the homology of the cover.  If  $J$ is the unknot, then the  2--fold branched cover is the lens space $L((2k+1)^2, 2k)$.

The homology group   $H_1(M_n(K(k,J)))$ can be computed using a formula of Seifert~\cite{MR0035436}; see~\cite{MR1201199} for a more recent treatment.  In our notation, this result states that for a knot $K$ with Seifert matrix $B$, $H_1(M_n(K))$ is presented by 
\[ 
\Gamma^n  - (\Gamma - \text{Id})^n,
\]
where $\Gamma = (B^{\sf T} - B)^{-1} B^{\sf T}$.

In our case, one readily computes that 
\[
\Gamma =    \begin{pmatrix}k+1 & -k \\
0 & -k \\
\end{pmatrix},
\]
and thus we are interested in the group presented by 
\[
A_k =   \begin{pmatrix}
k+1 & -k \\
0 & -k \\
\end{pmatrix}^n - 
\begin{pmatrix} k & -k \\
0 & -k-1 \\
\end{pmatrix} ^n .
\]
For some $b$, this is of the form 
\[A_k = 
\begin{pmatrix} 
(k+1)^n - k^n & b \\
0 &(-k)^n- (-k-1)^n\  \\
\end{pmatrix} .
\]
Since $n$ is odd, this can be rewritten as
\[A_k = 
\begin{pmatrix} 
(k+1)^n - k^n & b \\
0 &(k+1)^n - k^n\  \\
\end{pmatrix} .
\]
With a bit more work we could show that $b = 0$, but instead we rely on a theorem of Plans~\cite{MR0056923} (or see~\cite[Chapter 8D]{MR0515288}): the homology of an odd-fold cycle branched cover is a double. 
\end{proof}

\subsection{A number theoretic observation}

In our examples, we considered the cases of   $H_1(M_3(K(1,U))) \cong \Z_{7} \oplus \Z_7$ and  $H_1(M_3(K(2,U))) \cong \Z_{19} \oplus \Z_{19}$.  We observed that both $\F_7$ and $\F_{19}$ contain primitive $3$--roots of unity, since $7 \equiv 1 \mod 3 $ and $19 \equiv 1 \mod 3$.  This is not a coincidence.   Our examples were the cases of $p =3$ and either $k=1$ or $k=2$ in the following theorem, which follows  immediately from a standard application of the binomial theorem or from  Fermat's Little Theorem. 

\begin{theorem}  If $ p$ is prime, then for all   $k$,  $(k+1)^p - k^p \equiv 1 \mod p$.
\end{theorem}

\section{The eigenspace structure of $H_1(M_n(K), \F_p)$.}

In his survey paper on knot theory~\cite{MR521730}, Gordon used a duality argument to prove  that the first  homology of the infinite cyclic cover of a knot, viewed as a  module over the ring $\Z[\Z] \cong \Z[t, t^{-1}]$, is isomorphic to its dual module, in which the action of $t$ is replaced with the action of $t^{-1}$.  A similar argument can be applied in the setting of $n$--fold cyclic branched covers.  Here we give a simple proof of a consequence of such a result.  Duality is still required  to the extent that it implies that the linking form of a three-manifold is nonsingular. 

\begin{theorem}\label{thm:eigenequal} Assume that $H_1(M_n(K)) \cong \F_p^k$ for some $k$.  Suppose that $n$ divides $p-1$, so that $\F_p$ contains a primitive $n$--root of unity, $ \xi$.  Then $H_1(M_n(K))$ splits into a direct sum of  $\xi^i$--eigenspaces, denoted  $E_i$, under the action of the deck transformation $T_*$.   In addition, $E_0 $ is trivial and $E_i \cong E_{n-i}$ for all $i, 0 < i < n$.
\end{theorem}

\begin{proof}  Since  $T_*$ satisfies $T_*^n = 1$, the splitting into a direct sum of eigenspaces is an elementary fact from linear algebra. 

Let $ {\rm{lk}}(x,y) \in \F_p$ denote the $\F_p$--valued linking form on $H_1(M_n(K))$.   Recall that the linking form is symmetric, nonsingular and equivariant with respect to the action of a homeomorphism, in particular with respect to $T_*$.  \smallskip

\noindent{\bf Claim 1:}  The eigenspaces $E_i$ and $E_j$ are orthogonal with respect to the linking form unless $i= j = 0$ or $i = n- j$.  

To see this, suppose that $x \in E_i$ and $y \in E_j$.  Then 
\[
\xi^i  {\rm{lk}}(x,y) =  {\rm{lk}}(T_*x,y) =  {\rm{lk}}(x,T_*^{-1}y) =  {\rm{lk}}(x,\xi^{-j}y) =  \xi^{-j}{\rm{lk}}(x,y).
\]
It follows that  $
(\xi^i   -\xi^{-j} ) {\rm{lk}}(x,y)  =0 .$  This can be rewritten as  $
(\xi^i   -\xi^{n-j} ) {\rm{lk}}(x,y)  =0 .$  If $i \ne 0$, then $ \xi^i   -\xi^{n-j} \ne 0$ unless $i = n-j$.  Thus, if $i \ne 0$ and $i \ne n-j$, then $ {\rm{lk}}(x,y)  =0$.\smallskip

\noindent{\bf Claim 2:}  $E_{0} $ is trivial.  We can now write
\[
H_1(M(K)) \cong E_0  \oplus E_{n/2} \bigoplus_{1\le i < n/2} \big( E_i \oplus E_{n-i}\big) . 
\]
(The    summand $E_{n/2}$ exists if and only if $n$ is even, in which case it represents the $-1$--eigenspace.)  
  
  If $x \in E_0$, then $x + T_*x + \cdots + T_*^{n-1}x = n x$ is in the image of the transfer map $\tau \co H_1(S^3) \to H_1(M_n(K))$, and thus equals 0.  We can write $p-1 = nk$ for some $k$, and so $(p-1)x = 0$.  But $p-1$ is relatively prime to $p$, and so we have $x = 0$, as desired.

\smallskip 

\noindent{\bf Claim 3:}   $E_i \cong E_{n-i}$ for all $i, 0 < i < n$.

This is automatic for $E_{n/2}$  in the case the $n$ is even.  We focus on a summand    $E_i \oplus E_{n-i} $ for $1\le i < n/2$.

Suppose that $E_i$ is of dimension $a$ and  $E_{n-i}$ is of dimension $b$.  By choosing bases for these  eigenspaces, the linking form can be represented by an $(a+b) \times (a+b)$ matrix with entries in $\F_p$.  Both $E_i$ and $E_{n-i}$ are self-orthogonal, so there are blocks with all entries 0 of size $a \times a$ and $b \times b$.  The nonsingularity implies that $a \le (a+b) /2$ and $b \le (a+b)/2$.  This can occur   only if $a = b$.
\end{proof}


\bibliography{../BibTexComplete}

\begin{bibdiv}
\begin{biblist}

\bib{MR3307286}{article}{
      author={Abe, Tetsuya},
      author={Kanenobu, Taizo},
       title={Unoriented band surgery on knots and links},
        date={2014},
        ISSN={0289-9051},
     journal={Kobe J. Math.},
      volume={31},
      number={1-2},
       pages={21\ndash 44},
}

\bib{MR3825859}{article}{
      author={Aceto, Paolo},
      author={Golla, Marco},
      author={Lecuona, Ana~G.},
       title={Handle decompositions of rational homology balls and
  {C}asson-{G}ordon invariants},
        date={2018},
        ISSN={0002-9939},
     journal={Proc. Amer. Math. Soc.},
      volume={146},
      number={9},
       pages={4059\ndash 4072},
         url={https://doi.org/10.1090/proc/14035},
}

\bib{MR593626}{article}{
      author={Akbulut, Selman},
      author={Kirby, Robion},
       title={Branched covers of surfaces in {$4$}-manifolds},
        date={1979/80},
        ISSN={0025-5831},
     journal={Math. Ann.},
      volume={252},
      number={2},
       pages={111\ndash 131},
         url={https://mathscinet.ams.org/mathscinet-getitem?mr=593626},
}

\bib{MR968881}{article}{
      author={Bleiler, Steven~A.},
      author={Eudave Mu\~{n}oz, Mario},
       title={Composite ribbon number one knots have two-bridge summands},
        date={1990},
        ISSN={0002-9947},
     journal={Trans. Amer. Math. Soc.},
      volume={321},
      number={1},
       pages={231\ndash 243},
         url={https://doi.org/10.2307/2001599},
}

\bib{MR900252}{incollection}{
      author={Casson, A.~J.},
      author={Gordon, C.~McA.},
       title={Cobordism of classical knots},
        date={1986},
   booktitle={\`a la recherche de la topologie perdue},
      series={Progr. Math.},
      volume={62},
   publisher={Birkh\"auser Boston, Boston, MA},
       pages={181\ndash 199},
        note={With an appendix by P. M. Gilmer},
}

\bib{friedl2021homotopy}{article}{
      author={Friedl, Stefan},
      author={Kitayama, Takahiro},
      author={Lewark, Lukas},
      author={Nagel, Matthias},
      author={Powell, Mark},
       title={Homotopy ribbon concordance, blanchfield pairings, and twisted
  alexander polynomials},
        date={2021},
      eprint={arxiv.org/abs/2007.15289},
}

\bib{MR4179701}{article}{
      author={Friedl, Stefan},
      author={Powell, Mark},
       title={Homotopy ribbon concordance and {A}lexander polynomials},
        date={2020},
        ISSN={0003-889X},
     journal={Arch. Math. (Basel)},
      volume={115},
      number={6},
       pages={717\ndash 725},
         url={https://doi.org/10.1007/s00013-020-01517-5},
}

\bib{MR1201199}{article}{
      author={Gilmer, Patrick},
       title={Classical knot and link concordance},
        date={1993},
        ISSN={0010-2571},
     journal={Comment. Math. Helv.},
      volume={68},
      number={1},
       pages={1\ndash 19},
         url={https://doi.org/10.1007/BF02565807},
}

\bib{MR656619}{article}{
      author={Gilmer, Patrick~M.},
       title={On the slice genus of knots},
        date={1982},
        ISSN={0020-9910},
     journal={Invent. Math.},
      volume={66},
      number={2},
       pages={191\ndash 197},
         url={https://doi.org/10.1007/BF01389390},
}

\bib{MR1707327}{book}{
      author={Gompf, Robert~E.},
      author={Stipsicz, Andr\'as~I.},
       title={{$4$}-manifolds and {K}irby calculus},
      series={Graduate Studies in Mathematics},
   publisher={American Mathematical Society, Providence, RI},
        date={1999},
      volume={20},
        ISBN={0-8218-0994-6},
         url={https://doi.org/10.1090/gsm/020},
}

\bib{gong2020nonorientable}{article}{
      author={Gong, Sherry},
      author={Marengon, Marco},
       title={Non-orientable link cobordisms and torsion order in floer
  homologies},
        date={2020},
      eprint={arxiv.org/abs/2010.06577},
}

\bib{MR521730}{incollection}{
      author={Gordon, C.~McA.},
       title={Some aspects of classical knot theory},
        date={1978},
   booktitle={Knot theory ({P}roc. {S}em., {P}lans-sur-{B}ex, 1977)},
      series={Lecture Notes in Math.},
      volume={685},
   publisher={Springer, Berlin},
       pages={1\ndash 60},
         url={https://mathscinet.ams.org/mathscinet-getitem?mr=521730},
}

\bib{MR634459}{article}{
      author={Gordon, C.~McA.},
       title={Ribbon concordance of knots in the {$3$}-sphere},
        date={1981},
        ISSN={0025-5831},
     journal={Math. Ann.},
      volume={257},
      number={2},
       pages={157\ndash 170},
         url={https://doi.org/10.1007/BF01458281},
}

\bib{MR683753}{article}{
      author={Hartley, Richard},
       title={Identifying noninvertible knots},
        date={1983},
        ISSN={0040-9383},
     journal={Topology},
      volume={22},
      number={2},
       pages={137\ndash 145},
         url={https://mathscinet.ams.org/mathscinet-getitem?mr=683753},
}

\bib{MR3649355}{article}{
      author={Hendricks, Kristen},
      author={Manolescu, Ciprian},
       title={Involutive {H}eegaard {F}loer homology},
        date={2017},
        ISSN={0012-7094},
     journal={Duke Math. J.},
      volume={166},
      number={7},
       pages={1211\ndash 1299},
         url={https://doi-org.proxyiub.uits.iu.edu/10.1215/00127094-3793141},
}

\bib{hom2020ribbon}{article}{
      author={Hom, Jennifer},
      author={Kang, Sungkyung},
      author={Park, JungHwan},
       title={Ribbon knots, cabling, and handle decompositions},
        date={2020},
      eprint={arxiv.org/abs/2003.02832},
}

\bib{MR1075165}{article}{
      author={Hoste, Jim},
      author={Nakanishi, Yasutaka},
      author={Taniyama, Kouki},
       title={Unknotting operations involving trivial tangles},
        date={1990},
        ISSN={0030-6126},
     journal={Osaka J. Math.},
      volume={27},
      number={3},
       pages={555\ndash 566},
         url={http://projecteuclid.org/euclid.ojm/1200782446},
}

\bib{MR4186142}{article}{
      author={Juh\'{a}sz, Andr\'{a}s},
      author={Miller, Maggie},
      author={Zemke, Ian},
       title={Knot cobordisms, bridge index, and torsion in {F}loer homology},
        date={2020},
        ISSN={1753-8416},
     journal={J. Topol.},
      volume={13},
      number={4},
       pages={1701\ndash 1724},
         url={https://doi-org.proxyiub.uits.iu.edu/10.1112/topo.12170},
}

\bib{MR2755489}{article}{
      author={Kanenobu, Taizo},
       title={Band surgery on knots and links},
        date={2010},
        ISSN={0218-2165},
     journal={J. Knot Theory Ramifications},
      volume={19},
      number={12},
       pages={1535\ndash 1547},
         url={https://doi.org/10.1142/S0218216510008522},
}

\bib{MR1417494}{book}{
      author={Kawauchi, Akio},
       title={A survey of knot theory},
   publisher={Birkh\"{a}user Verlag, Basel},
        date={1996},
        ISBN={3-7643-5124-1},
        note={Translated and revised from the 1990 Japanese original by the
  author},
}

\bib{2019arXiv190412014K}{article}{
      author={{Kim}, Taehee},
      author={{Livingston}, Charles},
       title={{Knot reversal acts non-trivially on the concordance group of
  topologically slice knots}},
        date={2019-04},
     journal={arXiv e-prints},
      eprint={arxiv.org/abs/1904.12014},
}

\bib{MR4017212}{article}{
      author={McDonald, Clayton},
       title={Band number and the double slice genus},
        date={2019},
     journal={New York J. Math.},
      volume={25},
       pages={964\ndash 974},
}

\bib{MR2262340}{article}{
      author={Mizuma, Yoko},
       title={An estimate of the ribbon number by the {J}ones polynomial},
        date={2006},
        ISSN={0030-6126},
     journal={Osaka J. Math.},
      volume={43},
      number={2},
       pages={365\ndash 369},
  url={http://projecteuclid.org.proxyiub.uits.iu.edu/euclid.ojm/1152203945},
}

\bib{MR704925}{article}{
      author={Nakanishi, Yasutaka},
      author={Nakagawa, Yoko},
       title={On ribbon knots},
        date={1982},
        ISSN={0385-633x},
     journal={Math. Sem. Notes Kobe Univ.},
      volume={10},
      number={2},
       pages={423\ndash 430},
}

\bib{MR2065507}{article}{
      author={Ozsv\'{a}th, Peter},
      author={Szab\'{o}, Zolt\'{a}n},
       title={Holomorphic disks and knot invariants},
        date={2004},
        ISSN={0001-8708},
     journal={Adv. Math.},
      volume={186},
      number={1},
       pages={58\ndash 116},
         url={https://mathscinet.ams.org/mathscinet-getitem?mr=2065507},
}

\bib{MR0056923}{article}{
      author={Plans, Antonio},
       title={Contribution to the study of the homology groups of the cyclic
  ramified coverings corresponding to a knot},
        date={1953},
     journal={Revista Acad. Ci. Madrid},
      volume={47},
       pages={161\ndash 193 (5 plates)},
}

\bib{MR0345089}{book}{
      author={Reidemeister, K.},
       title={Knotentheorie},
   publisher={Springer-Verlag, Berlin-New York},
        date={1974},
        note={Reprint},
}

\bib{MR717222}{book}{
      author={Reidemeister, K.},
       title={Knot theory},
   publisher={BCS Associates, Moscow, Idaho},
        date={1983},
        ISBN={0-914351-00-1},
        note={Translated from the German by Leo F. Boron, Charles O.
  Christenson and Bryan A. Smith},
}

\bib{MR0515288}{book}{
      author={Rolfsen, Dale},
       title={Knots and links},
   publisher={Publish or Perish, Inc., Berkeley, Calif.},
        date={1976},
        note={Mathematics Lecture Series, No. 7},
}

\bib{Sarkar_2020}{article}{
      author={Sarkar, Sucharit},
       title={Ribbon distance and khovanov homology},
        date={2020Apr},
        ISSN={1472-2747},
     journal={Algebraic \& Geometric Topology},
      volume={20},
      number={2},
       pages={1041\ndash 1058},
         url={http://dx.doi.org/10.2140/agt.2020.20.1041},
}

\bib{MR0035436}{article}{
      author={Seifert, H.},
       title={On the homology invariants of knots},
        date={1950},
        ISSN={0033-5606},
     journal={Quart. J. Math., Oxford Ser. (2)},
      volume={1},
       pages={23\ndash 32},
         url={https://mathscinet.ams.org/mathscinet-getitem?mr=0035436},
}

\bib{MR4024565}{article}{
      author={Zemke, Ian},
       title={Knot {F}loer homology obstructs ribbon concordance},
        date={2019},
        ISSN={0003-486X},
     journal={Ann. of Math. (2)},
      volume={190},
      number={3},
       pages={931\ndash 947},
         url={https://doi.org/10.4007/annals.2019.190.3.5},
}

\end{biblist}
\end{bibdiv}
\bibliographystyle{plain}	

\end{document}